\documentclass[11pt]{article}

\usepackage[english]{babel}
\usepackage{amsmath,amssymb,amsthm}
\usepackage[pdftex]{graphicx,color}

\newtheorem{theorem}{Theorem}[section]
\newtheorem{proposition}[theorem]{Proposition}
\newtheorem{lemma}[theorem]{Lemma}
\newtheorem{corollary}[theorem]{Corollary}

\theoremstyle{definition}
\newtheorem{definition}[theorem]{Definition}

\newtheorem*{remark}{Remark}

\newcommand{\Z}{\mathbb{Z}}

\newcommand{\R}{\mathbb{R}}
\newcommand{\Rplus}{\mathbb{R}_+}

\newcommand{\C}{\mathbb{C}}

\renewcommand{\H}{\mathbb{H}}

\newcommand{\dist}{\operatorname{dist}}

\newcommand{\prob}[1]{\mathrm{\textbf{P}} \left ( #1 \right)}

\newcommand{\Qxprob}[1]{\mathrm{\textbf{Q}}_x \left( #1 \right)}
\newcommand{\expect}[1]{\mathrm{\textbf{E}} \left[ #1 \right]}

\newcommand{\condexpect}[2]{\mathrm{\textbf{E}} \left[ \left. #1 \right | #2 \right]}
\newcommand{\indicate}[1]{\mathbf{1} \left \{ #1 \right \}}
\newcommand{\Ft}{\mathcal{F}_t}
\newcommand{\F}{\mathcal{F}}
\newcommand{\hitpoints}{\gamma \cap \mathbb{R}}
\newcommand{\hitpointst}{\gamma[0,t] \cap \Rplus}
\newcommand{\Henergy}{\mathcal{E}}
\newcommand{\te}[1]{#1_t^{\epsilon}}
\newcommand{\xe}[1]{#1_x^{\epsilon}}
\newcommand{\ep}[1]{#1^{\epsilon}}
\newcommand{\haussdim}[0]{\mathrm{dim_H} \,}

\newcommand{\borelplus}{\mathcal{B} \left( \Rplus \right)}
\newcommand{\mum}[1]{\mu \left( #1 \right)}
\newcommand{\leb}[1]{\left| #1 \right|}

\newcommand{\dtrans}[2]{#2_{d, #1}}

\begin{document}

\title{The Covariant Measure of SLE on the Boundary}
\author{\begin{tabular}{cc}
Tom Alberts\footnote{Research supported in part by NSF Grant OISE 0730136.} & Scott Sheffield\footnote{Research supported in part by NSF Grants DMS 0403182, DMS 064558 and OISE 0730136.} \\
\small{Department of Mathematics} & \small{Department of Mathematics} \\
\small{University of Toronto} & \small{Massachusetts Institute of Technology} \\
\small{Toronto, ON, Canada} & \small{Cambridge, MA, USA} \\
\small{E-mail: \texttt{alberts@math.toronto.edu}} & \small{E-mail: \texttt{sheffield@math.mit.edu}} \\
\end{tabular}
\\}
\date{}

\maketitle

\begin{abstract}
We construct a natural measure $\mu$ supported on the intersection of a chordal SLE$(\kappa)$ curve $\gamma$ with $\R$, in the range $4 < \kappa < 8$. The measure is a function of the SLE path in question.  Assuming that boundary measures transform in a ``$d$-dimensional'' way (where $d$ is the Hausdorff dimension of $\gamma \cap \R$), we show that the measure we construct is (up to multiplicative constant) the unique measure-valued function of the SLE path that satisfies the Domain Markov property.
\end{abstract}

\section{Introduction \label{Intro}}

\subsection{Statement of Main Result}

Let $\gamma$ be a chordal SLE$(\kappa)$ curve in $\mathbb{H}$ with
$4 < \kappa < 8$. In this range of $\kappa$ it is well known that the intersection of $\gamma$
with the real line is a random set with fractional dimension, and in the
recent papers \cite{alberts_sheff:dimension} and
\cite{schramm_zhou:dimension} it has been shown that the Hausdorff
dimension of $\hitpoints$ is almost surely $d := 2-8/\kappa$. This result
gives some information on the size of $\hitpoints$, but it is
still only a qualitative description. Two instances of
$\hitpoints$ will be very different as point sets, even
though their Hausdorff dimension will be the same. The purpose of
this paper is to gain more information on the local structure of
$\hitpoints$ by constructing a measure on it, somewhat analogous to the local time measure on the zeros of a Brownian motion. That is, we define a measure-valued function $\mu$ of $\gamma$ which is almost surely supported on $\gamma \cap \R$.  For each finite interval $I$, we will interpret the quantity $\mu(I)$ as a kind of $d$-dimensional ``volume'' of $\gamma \cap I$. Although we will not prove that the measure we construct is equivalent to the most well known notions of $d$-dimensional volume (for example, it remains an open question whether it is equivalent to $d$-dimensional Minkowski measure on $\gamma \cap \R$, or is some sort of Hausdorff content), we will see that it it is the unique $d$-dimensional volume measure on $\gamma \cap \R$ with certain natural conformal covariance properties.

The construction of our measure is inspired by a similar work in progress by the second author and Greg Lawler (see \cite{lawler_sheff:time}), but many of the ideas can already be found in \cite{lawler:curvedim}. Although in spirit our work is very similar to the work in progress by Lawler and Sheffield, there are differences that we should highlight. First, the measure they are constructing is supported on the SLE($\kappa$) curve in the interior of the domain, while ours deals only with the intersection of the curve with the boundary in the range $4 < \kappa < 8$. Second, although most of the content of \cite{lawler_sheff:time} is geared towards constructing a measure supported on $\gamma$, they interpret this measure as a ``natural time parameterization'' for the curve and most of their work is draped in the language of the time parameterization. Finally, the technical lemmas they use to establish the existence of a non-trivial measure are very different from the ones we use. In both this work and \cite{lawler_sheff:time}, the key idea in the construction of the measure is an appeal to the Doob-Meyer decomposition theorem, which gives a unique way of writing a supermartingale as a local martingale minus a predictable, non-decreasing process (all processes here are assumed to be cadlag; see Section \ref{DMSection} for a precise definition of a predictable process). Our work builds on a recent ``two-point martingale'' discovered by Schramm and Zhou (\cite{schramm_zhou:dimension}) and uses the well established theory of the Doob-Meyer decomposition to give a clean and direct proof of the existence of the measure. Lawler and Sheffield, on the other hand, have no such two-point martingale available to them and they are forced to reinvent much of the Doob-Meyer theory from scratch.

The intuition behind the use of Doob-Meyer is very simple. Let $K_t$ be the SLE hull at time $t$. As the curve goes from $0$ to $\infty$ it swallows all of the points on the positive real axis $\Rplus$. Since SLE curves are non-self-crossing, after an interval $I$ has been swallowed it is impossible for the curve to return to it. As we stated earlier, our intention is that the measure of $I$ will in some sense describe the $d$-dimensional volume of the set $\gamma \cap I$. Hence, it is reasonable to assume that the measure of $I$ is completely determined at its swallowing time, which is random but still a stopping time. At any time $t \geq 0$ we can always decompose $I$ into a left interval $I \cap K_t$ of swallowed points and a right interval $I \backslash K_t$ of unswallowed points, and if $\mu$ is any Borel measure on $\Rplus$ then
\begin{align} \label{muDecomp1}
\mum{I} = \mum{I \cap K_t} + \mum{I \backslash K_t}.
\end{align}
Assuming that $\mum{I \cap K_t}$ is measurable with respect to the filtration generated by $\gamma[0,t]$, taking conditional expectations of both sides of \eqref{muDecomp1} gives
\begin{align}\label{muDecomp2}
\condexpect{\mu(I)}{\gamma[0,t]} = \mum{I \cap K_t} + \condexpect{\mum{I \backslash K_t}}{\gamma[0,t]}.
\end{align}
The process $\mum{I \cap K_t}$ is non-decreasing with $t$, since $K_t$ is, and assuming that $\mu(I)$ has finite expectation, the left hand side of \eqref{muDecomp2} is a martingale. Consequently, $\condexpect{\mum{ I \backslash K_t }}{\gamma[0,t]}$ must be a supermartingale, and \eqref{muDecomp2} is its Doob-Meyer decomposition. This gives a strategy for constructing $\mu$. The idea is to identify a natural, explicit formula for the supermartingale, and then define the process $\mum{I \cap K_t}$ to be the non-decreasing part of its Doob-Meyer decomposition. Taking $t \to \infty$ gives $\mu(I)$, since all points are eventually swallowed, and then repeating the process for all intervals $I$ \textit{uniquely} determines $\mu$.

Our choice of $\condexpect{\mum{ I \backslash K_t }}{\gamma[0,t]}$ will ultimately be determined by a requirement that the boundary measure we construct satisfies a certain Domain Markov property. Suppose we are given $\gamma[0,t]$ (and the corresponding hull $K_t$) and let $h_t : \H \backslash K_t \to \H$ be the unique conformal map such that $h_t(\gamma(t)) = 0$, $h_t(\infty) = \infty$, and $h_t(z) \sim z$ as $z \to \infty$. The usual Domain Markov property for SLE says that the image curve
\begin{align*}
\gamma^t(s) := h_t(\gamma(t+s)), \,\, s \geq 0
\end{align*}
is independent of $\gamma[0,t]$ and has the SLE law in $\H$ from $0$ to $\infty$. A similar statement should also hold for the boundary measure.  To formulate this statement, we first need to articulate how we wish our measure to transform under a conformal map $\phi$.  Just as ``lengths'' are locally stretched by a factor of $|\phi'|$ and ``areas'' by a factor of $|\phi'|^2$, it is natural to expect the measure of a $d$-dimensional set to be locally changed by a factor of $|\phi'|^d$.

\begin{definition}\label{transformDefn}
Let $D_1$ and $D_2$ be simply connected domains and let $\phi$ be a conformal map of $D_1$ onto $D_2$. Assume that $\phi'$ extends continuously to all of $\partial D_1$. Let $\nu$ be a measure supported on $\partial D_1$. The \textbf{d-dimensional covariant transform} of $\nu$ by $\phi$ is the measure $\dtrans{\phi}{\nu}$ on $\partial D_2$ defined by
\begin{align*}
\dtrans{\phi}{\nu}(\phi(A)) := \int_A \leb{\phi'(w)}^d \, d\nu(w)
\end{align*}
for all Borel subsets $A$ of $\partial D_1$. We will also use the convenient infinitesimal shorthand
\begin{align*}
d \dtrans{\phi}{\nu}(\phi(w)) = \leb{\phi'(w)}^d d\nu(w).
\end{align*}
If $\phi'$ does not extend continuously to all of $\partial D_1$ then we restrict the definition to Borel sets $A$ that are contained within some boundary segment $S \subset \partial D_1$ to which $\phi'$ does extend continuously.
\end{definition}

With this definition in hand we state our main result:

\begin{theorem}\label{mainTheorem}
Let $\mathrm{\textbf{P}}$ be the chordal SLE($\kappa$) measure on curves in $\H$ from $0$ to $\infty$.  Up to redefinition on a set of $\mathrm{\textbf{P}}$-measure zero, there is a unique measure-valued function $\mu$ on curves $\gamma$ (defined as a measure on $\R_+$ for $\mathrm{\textbf{P}}$ almost all $\gamma$) with the following properties:
\begin{enumerate}
\item {\bf Scaling:} $\mu(r \cdot)$ has the same law as $r^d \mu(\cdot)$.
\item {\bf Finite Expectation/Normalization:}  $\expect{\mum{(0,1]}} = 1/d$.
\item {\bf Measurability:} The process given by restricting $\mu$ to $K_t$, i.e. $\mum{\cdot \cap K_t}$, is predictable.
\item {\bf Domain Markov Property:} Given $\gamma[0,t]$, the conditional law of $\mu_{d, h_t}$ (where we have restricted $\mu$ to $\Rplus \backslash K_t$, on which $h_t$ is smooth) is the same as the original law of $\mu$ restricted to $h_t(\Rplus \backslash K_t)$.
\end{enumerate}
In addition, the measure almost surely
\begin{enumerate}
\renewcommand{\labelenumi}{(\roman{enumi})}
\item is supported on $\gamma \cap \Rplus$,
\item is singular with respect to Lebesgue measure,
\item is free of atoms,
\item and assigns positive mass to an open interval if and only if the curve hits that interval.
\end{enumerate}
Moreover, for any interval $I \subset \Rplus$
\begin{align*}
\expect{\mum{I \cap K_t}} = \frac{1}{\Gamma \left(\frac{12 - \kappa}{2 \kappa} \right)} \int_I \int_{x^2/2t}^{\infty} x^{d-1} u^{\frac{12 - \kappa}{2\kappa}-1} e^{-u} \, du \, dx.
\end{align*}
\end{theorem}

The uniqueness part of Theorem \ref{mainTheorem} is essentially guaranteed by the uniqueness of the Doob-Meyer decomposition, provided that we can find an explicit supermartingale to take the place of $\condexpect{\mum{I \backslash K_t}}{\gamma[0,t]}$. In fact the extra scaling, normalization and Domain Markov properties of $\mu$ uniquely determine what $\condexpect{\mum{I \backslash K_t}}{\gamma[0,t]}$ must be. The scaling implies that $\expect{\mum{rI}} = r^d \expect{\mum{I}}$ for any interval $I \subset \Rplus$; from this we can deduce that there exists a constant $C > 0$ such that $\expect{\mum{I}} = C \int_I x^{d-1} \, dx$, or in the infinitesimal shorthand $\expect{d \mu(x)} = C x^{d-1} \, dx$. The normalization implies that $C=1$. A corollary of the Domain Markov property is that the conditional expectation of $\mu_{d,h_t}$ given $\gamma[0,t]$ is the same as the a priori expectation of $\mu$; that is
\begin{align*}
\condexpect{d \dtrans{h_t}{\mu}(h_t(x))}{\gamma[0,t]} = h_t(x)^{d-1} \, d(h_t(x)) = h_t(x)^{d-1} h_t'(x) \, dx,
\end{align*}
for unswallowed points $x \in \R_+ \backslash K_t$. But by definition of the $d$-dimensional covariant transform, this is the same as
\begin{align*}
\condexpect{h_t'(x)^d d\mu(x)}{\gamma[0,t]} = h_t'(x)^{d-1} h_t'(x) \, dx
\end{align*}
which, since $h_t'(x)$ is measurable with respect to the information in $\gamma[0,t]$, implies that
\begin{align*}
\condexpect{d\mu(x)}{\gamma[0,t]} = \left( \frac{h_t'(x)}{h_t(x)} \right)^{1-d} \, dx.
\end{align*}
To simplify notation we will write
\begin{align*}
M_t(x) := \left( \frac{h_t'(x)}{h_t(x)} \right)^{1-d}, \quad X_t(I) := \int_{I \backslash K_t} M_t(x) \, dx.
\end{align*}
In Section \ref{MtSection} we will see that $M_t(x)$ is actually a positive local martingale for each $x \in \Rplus$, and that for each $x$ it is almost surely the case that $M_t(x) \to 0$ as $t \uparrow T_x$, where $T_x$ is the swallowing time of $x$ by the SLE hull. If we define $M_t(x)$ to be zero for $t \geq T_x$, then we have
\begin{align*}
X_t(I) = \int_I M_t(x) \, dx.
\end{align*}
In Section \ref{ConstructSection} we will see that $X_t(I)$ is actually a supermartingale. By the argument outlined earlier we can use its non-decreasing process to construct the measure $\mu$. The difficult part of the construction is in showing that $X_t(I)$ actually admits a decomposition as a \textit{martingale} minus a predictable, non-decreasing process. This form of the decomposition is \textit{not} guaranteed for any positive supermartingale. What is true, as we will describe in more detail in Section \ref{DMSection}, it that a positive supermartingale can always be uniquely decomposed as a \textit{local martingale} minus a predictable, non-decreasing process. The latter decomposition, however, can be ``trivial'' in the sense that the non-decreasing process may turn out to be identically zero. In fact, one surprising result of this paper is that while for each $x \in \R_+$ the process $M_t(x)$ is itself a local martingale --- and hence admits only a trivial decomposition --- the integral $X_t(I)$ admits a non-trivial decomposition. This illustrates the fact that an average of non-negative local martingales with respect to a given filtration need not be a local martingale (even though an average of non-negative {\it martingales} is a martingale by Fubini's theorem).

The details of constructing the Doob-Meyer decomposition for $X_t(I)$ are the subject of Section \ref{ConstructSection}. To mention some keywords, we first must show that the process $X_t(I)$ satisfies a uniform integrability condition called the \textit{class $\mathcal{D}$} property. This is a technical condition that \textit{a priori} seems difficult to verify, but actually turns out to be easily satisfied thanks to a recent ``two-point'' martingale discovered by Schramm and Zhou. After establishing the \textit{class $\mathcal{D}$} property we also show that $X_t(I)$ is \textit{regular}. This additional property guarantees that the increasing part corresponding to $X_t(I)$ is continuous which, as we will see in Section \ref{muProperties}, proves that our boundary measure is free of atoms.

In Section \ref{approxSection} we discuss the relationship between our work and the work of Schramm and Zhou in \cite{schramm_zhou:dimension}, which motivated much of this paper. Schramm and Zhou also used $M_t(x)$ to construct what is implicitly the same boundary measure as ours, although they constructed it only as a $\gamma$-dependent subsequential limit of approximating measures $\ep{\mu}$ and were only able to show convergence on an event of positive probability (less than one). Via another abstract appeal to the Doob-Meyer decomposition, we are able to strengthen the method of taking a limit of the $\ep{\mu}$ and show that there is a fixed subsequence $\epsilon_j$ tending to zero along which $\mu^{\epsilon_j}$ converges weakly to $\mu$, with probability one. Details of the latter point occupy Section \ref{epSection}. In Section \ref{conformalMinkSection} we show how the $\ep{\mu}$ measures lead to a natural conjecture that our boundary measure on $\hitpoints_+$ is in fact the same as the Minkowski measure, up to some fixed deterministic constant.

The rest of the paper is organized as follows: in Section \ref{NotationSect} we briefly set up notation and establish a few preliminary facts. In Section \ref{MtSection} we discuss the local martingale $M_t(x)$ and describe many of its properties, including the two-point bound of Schramm and Zhou which is essential to our work. In Section \ref{DMSection} we lay out the statement of the Doob-Meyer decomposition and associated lemmas that we will use to construct our measure. Section \ref{ConstructSection} is mostly devoted to verifying that $X_t(I)$ satisfies the hypotheses of the theorems listed in Section \ref{DMSection}, and then using the non-decreasing part of $X_t(I)$ to construct our boundary measure. In Section \ref{muProperties} we prove that the boundary measure of Section \ref{ConstructSection} has all the properties of Theorem \ref{mainTheorem}, including the important Domain Markov property. Section \ref{approxSection} discusses the relationship between our boundary measure and the approximate measures of Schramm and Zhou.

\section{Notation and Preliminaries \label{NotationSect}}

First we briefly establish some notation. By the symmetry of the SLE curve about the imaginary axis it is
enough to construct our measure only on the positive reals, which we
will denote by $\Rplus$. Given a Borel set $A \in \borelplus$, we will
use $\left| A \right|$ to denote the Lebesgue measure of $A$.

Throughout this paper it will be convenient to work with a single abstract probability space $(\Omega, \F, P)$. Our heavy usage of the Doob-Meyer decomposition requires many precise statements about martingales, and it is easier to deal with a single probability measure and filtration everywhere. Let $\{B_t; t \geq 0 \}$ be a one-dimensional standard Brownian motion on our probability space with $B_0 = 0$ almost surely, and let $\F_t$ be the filtration it generates (where $\F_0$ is augmented to include all $P$-negligible sets). With this definition $\Ft$ satisfies what are commonly called the \textit{usual conditions} in the literature on stochastic processes (i.e., $\F_t$ is right continuous and $\F_0$ contains all the $P$-negligible events in $\F$; for a further explanation and a precise definition of right continuity the interested reader can consult \cite[Chapter 1, Definition 2.25]{karatzas_shreve}).  For our purposes it is sufficient to know that the \textit{usual conditions} are a technical requirement of the standard Doob-Meyer decomposition theorem (as well as many other standard theorems in stochastic calculus).

\subsection{SLE and Bessel Processes}

Chordal SLE on $(\Omega, \F, P)$ is constructed via the Loewner equation
\begin{align}\label{SLEeqn}
\partial_t g_t(z) = \frac{a}{g_t(z) - U_t}, \,\, g_0(z) = z,
\end{align}
where $a = 2/\kappa$ and $U_t = -B_t$. This differs from the standard notation introduced in \cite{schramm:LERW} but is convenient since many of our parameters are in terms of $1/\kappa$ rather than $\kappa$. Let $\gamma$ be the SLE curve generated by $g_t$. In many instances we will also use $\gamma$ as shorthand for the trace $\gamma[0, \infty)$; it will be clear from the context which we are referring to.

In this paper we will only consider the range $1/4 < a < 1/2 \quad (4 < \kappa < 8)$. For each $z \in \mathbb{H}$, let $T_z$ be the stopping time at which the solution to \eqref{SLEeqn} explodes. It is well known in SLE that for $\kappa > 4$, $T_z$ is finite almost surely for each $z \in \overline{\H}$. The SLE hull $K_t$ is defined as
\begin{align*}
K_t := \overline{\{ z \in \overline{\H} : T_z \leq t\}}.
\end{align*}
At any time $t > 0$, $K_t \cap \R$ is a closed interval with zero in its interior. Note that $K_t \cap \R$ is \textit{not} the same as $\gamma[0,t] \cap \R$ : the former is the set of points swallowed by the curve and the latter is the set of points hit by the curve. In fact, we have the strict inclusion
\begin{align*}
\gamma[0,t] \cap \R \subset K_t \cap \R
\end{align*}
since hit points are swallowed points but not necessarily vice-versa.

For ease of notation, we will write $h_t$ for the shifted maps
\begin{align*}
h_t(z) = g_t(z) - U_t.
\end{align*}
For fixed $t > 0$, both $g_t$ and $h_t$ are analytic on $\overline{\H} \backslash K_t$. In fact $h_t$ is the unique conformal map from $\H \backslash K_t$ to $\H$ such that $h_t(\gamma(t)) = 0$, $h_t(\infty) = \infty$, and $h_t(z) \sim z$ as $z \to \infty$. From $U_t = -B_t$ and \eqref{SLEeqn} it follows that
\begin{align}\label{htdiff}
dh_t(z) = \frac{a}{h_t(z)} \, dt + dB_t, \,\, h_0(z) = z.
\end{align}
By convention we set $h_t(0) = 0$ for all $t \geq 0$. For each $x \in \Rplus$, the stochastic differential equation \eqref{htdiff} identifies $h_t(x)$ as a Bessel process on the real line, and the solution $h_t$ can be regarded as a stochastic flow on $\Rplus$. Note that
\begin{align*}
T_z = \inf \{ t \geq 0 : h_t(z) = 0 \}.
\end{align*}

Occasionally we will need to make use of the flow started from a later time, so for $s, t \geq 0$ we define $B_{t,s} := B_{t+s} - B_t$, $\F_{t,s} := \sigma(B_{t,s}; s \geq 0)$, and the flow $h_{t,s}$ by
\begin{align*}
dh_{t,s}(z) = \frac{a}{h_{t,s}(z)} \, ds + dB_{t,s}, \,\, h_{t,0}(z) = z,
\end{align*}
where all differentials are with respect to $s$. Again we adopt the convention that $h_{t,s}(0) = 0$ for all $s \geq 0$. The Markov property for the Brownian motion implies that the process
\begin{align*}
\left \{ h_{t,s}(z); s \geq 0, z \in \overline{\H} \right \}
\end{align*}
is independent of $\Ft$ and has the same law as the process $\left \{ h_t(z); t \geq 0, z \in \overline{\H} \right \}$. It is clear from the definitions of $h_t$ and $h_{t,s}$ that
\begin{align}\label{hCommute}
h_{t+s}(z) = h_{t,s}(h_t(z))
\end{align}
for all $z \in \overline{\H}$ and all $s, t \geq 0$. We further define
\begin{align*}
\gamma^t(s) := h_t(\gamma(t+s)), \quad K_{t,s} := h_t(K_{t+s}), \,\, s \geq 0,
\end{align*}
and then $h_{t,s}$ can be characterized as the unique conformal map from $\H \backslash K_{t,s}$ onto $\H$ such that $h_{t,s}(\gamma^t(s)) = 0$, $h_{t,s}(\infty) = \infty$ and $h_{t,s}(z) \sim z$ as $z \to \infty$.

From now on we will only be interested in the Bessel flow restricted to $\Rplus$. There are some basic properties of the flow that we will use repeatedly. We list them below. For a general survey of Bessel processes see \cite[Chapter XI]{revuz_yor}, and for more on flows arising from stochastic differential equations see \cite{kunita:book}.

\begin{proposition} \label{BesselProp}
The following are true:
\begin{enumerate}
\addtolength{\itemsep}{-0.5\baselineskip}
\renewcommand{\labelenumi}{(\alph{enumi})}
\item The processes $\{h_t(x); t \geq 0, x > 0 \}$ and $\{r^{-1}h_{r^2t}(rx); t \geq 0, x > 0\}$ have the same law.
\item For $a \geq 1/2 \,\, (\kappa \leq 4)$  we have $\prob{T_x = \infty \textnormal{ for all } x \in \Rplus} = 1$.
\item For $a < 1/2 \,\, (\kappa > 4)$ we have $\prob{T_x < \infty \textnormal{ for all } x \in \Rplus} = 1$.
\item If $0 < x < y$, then $0 < h_t(x) < h_t(y)$ for all $t < T_x$.
\item If $0 < x < y$, then $T_x \leq T_y$.
\item For $a > 1/4 \,\, (\kappa < 8)$ and $0 < x < y$, the event $T_x = T_y$ has positive probability.
\item For $a \leq 1/4 \,\, (\kappa \geq 8)$ and $0 < x < y$, the event $T_x = T_y$ has probability zero.
\item For $a > 0$, we have $0 \leq h_t'(x) \leq 1$ for all $x \in \Rplus$ and $t < T_x$.
\end{enumerate}
\end{proposition}

Statement (a) says that Bessel processes have the same scaling property as Brownian motion, which is easily checked. Statements (b) and (c) tell us in what ranges of $a$ the flow sends points to zero, which in the SLE context corresponds to the values of $\kappa$ for which the curve can hit the real line. Statement (d) says that the flow preserves the order on the real line, while statement (e) says that $T_x$ is a non-decreasing function of $x$. Statement (f) is very important for this paper. For the Bessel flow it says that two distinct points on $\Rplus$ can be sent to zero at the same time, which in the SLE context means that the points can be swallowed by the curve at the same time, at least in the range $4 < \kappa < 8$. In fact there is an exact expression for $\prob{T_x = T_y}$; see \eqref{OneIntHittingProb} for the exact formula and \cite[Proposition 6.34]{lawler:book} for details of the computation. Statement (g) says that in the $\kappa \geq 8$ range the curve hits points on $\Rplus$ individually. Statement (h) says that the flow is always a contraction on $\Rplus$. Proofs of (b)-(h) can be found in \cite[Proposition 1.21]{lawler:book}.

The dimension $d := 2 - 8/\kappa = 2 - 4a$ of $\hitpoints$ will be
used throughout the paper, as will $\beta := 1 - d = 4a - 1$. Note
that $0 < d < 1$ and $0 < \beta < 1$. For $h_t$ satisfying \eqref{htdiff} and $x \in \Rplus$, a central fact of our paper is that
\begin{align*}
M_t(x) := \left( \frac{h_t'(x)}{h_t(x)} \right)^{\beta}
\end{align*}
is a local martingale for $t \in [0, T_x)$, with $M_0(x) = x^{-\beta}$. In fact, using \eqref{htdiff} and Ito's Lemma it is easily verified that
\begin{align}\label{MDiff}
\frac{dM_t(x)}{M_t(x)} = \frac{-\beta}{h_t(x)} \, dB_t, \quad t < T_x.
\end{align}
Observe that since $0 \leq h_t'(x) \leq 1$ for all $t < T_x$ (Proposition \ref{BesselProp}(h)), we have the basic inequality
\begin{align}\label{MtIneq}
M_t(x) \leq h_t(x)^{-\beta}
\end{align}
for all $t < T_x$. We will make use of this in Section \ref{ConstructSection}.

We also define $M_t(x)$ started from a later time by
\begin{align*}
M_{t,s}(x) := \left( \frac{h_{t,s}'(x)}{h_{t,s}(x)} \right)^{\beta}.
\end{align*}
Again the Markov property of Brownian motion implies that $\{ M_{t,s}(x); s \geq 0, x > 0 \}$ is independent of $\Ft$ and has the same law as the process $\{M_t(x); t \geq 0, x > 0 \}$. An immediate consequence of relation \eqref{hCommute} is
\begin{align}\label{mCommute}
M_{t+s}(x) = M_{t,s}(h_t(x))h_t'(x)^{\beta}.
\end{align}
The latter relation will be important in Section \ref{muProperties} for proving that the measure we will construct has the Domain Markov property.

We will usually be interested in $M_t(x)$ when it grows large, so we define the stopping time
\begin{align*}
\xe{T} = \inf \{ t \geq 0 : M_t(x) \geq \epsilon^{-\beta} \} \wedge T_x.
\end{align*}
Note that $\xe{T} \leq T_x$, with strict inequality if and only if $M_t(x) \geq \epsilon^{-\beta}$ for some $t < T_x$. We also define
\begin{align*}
\te{M}(x) = M_{t \wedge \xe{T}}(x).
\end{align*}
By definition, $\te{M}(x)$ is a bounded local martingale and therefore a martingale. We will also need to keep track of points $x$ for which $M_t(x)$ grows large, so for $t \geq 0$ and $\epsilon > 0$ we define the random sets
\begin{align*}
\te{C} := \left \{ x \in \Rplus : \sup_{0 \leq s \leq t} M_s(x) \geq
\epsilon^{-\beta} \right \}
\end{align*}
Note that for each $\epsilon > 0$, $\te{C}$ is $\Ft$-measurable. Moreover $\te{C}$ increases as $t$ increases and decreases as
$\epsilon$ decreases, so we define
\begin{align*}
C^{\epsilon} := \bigcup_{t \geq 0} \te{C}, \quad C_t :=
\bigcap_{\epsilon > 0} \te{C},
\end{align*}
and
\begin{align*}
C := \bigcup_{t \geq 0} C_t = \bigcap_{\epsilon > 0} C^{\epsilon}.
\end{align*}
Note that the last equality is not completely trivial, but only uses that $\lim_{\epsilon \downarrow 0} T_x^{\epsilon} = T_x < \infty$.

The sets $\te{C}$ are meant to provide approximations to $\hitpointst$, but are nicer to work with because membership in $\te{C}$ is determined by the behavior of a family of martingales. These martingales are easier to analyze than the SLE curve itself. We will see in Section \ref{MtSection} that the approximation is good in that each element of $\te{C}$ is at most distance $4 \epsilon$ from $\gamma[0,t]$, and that each $C_t$ is a subset of $\hitpointst$. Similar sets first appeared in \cite{schramm_zhou:dimension}, although they only considered $\ep{C}$ and did not let the sets vary with time.

\subsection{Random Measures on $\Rplus$}

Throughout this paper we will freely use the term ``random measures on $\Rplus$''. There is room for misinterpretation as to what precisely this means, so we use this section to give an exact definition.

Let $\mathcal{M}$ be the space of positive Borel measures on $\Rplus$. By a random measure on $\Rplus$ we mean an $\F$-measurable function $\mu : \Omega \to \mathcal{M}$. Usually when working with a random measure we do \textit{not} explicitly write the dependence of $\mu$ on $\omega \in \Omega$, which we feel keeps the notation simpler and cleaner. For $A \in \borelplus$ we regard $\mu(A)$ as a random variable $\mu(A) : \Omega \to \Rplus$. On the rare occasion where we wish to write the measure of a particular set $A$ for a particular $\omega \in \Omega$ we will write $\mu(A)(\omega)$.

Recall that to construct a deterministic Borel measure $\nu$ on $\Rplus$ it is enough to define it as a countably additive set function on a field that generates $\borelplus$. The Carath\'{e}odory Extension Theorem (see \cite[Theorem 1.1]{varadhan:cln}) then allows one to uniquely lift $\nu$ to a countably additive measure defined on all of $\borelplus$. To construct a random measure one may use the above procedure for all (or almost all) $\omega \in \Omega$.

\section{The Local Martingale $M_t(x)$ and the Sets $\te{C}$ \label{MtSection}}

\subsection{The Local Martingale $M_t(x)$}

In this section we study the local martingale $M_t(x)$ introduced in Section \ref{Intro}. One reason that this local martingale is natural is that, in a sense that we will make precise below, $M_t(x)$
describes the conditional probability of the point $x$ being hit by
the SLE curve, given the curve up to time $t$. Indeed,
\cite{rohde_schramm} proves that for $\epsilon > 0$
\begin{align}
F(\epsilon) := \prob{\gamma \cap [1,1+\epsilon] \neq \emptyset} =
\frac{\Gamma(2a)}{\Gamma(1-2a) \Gamma(4a-1)} \int_0^{\epsilon}
\frac{du}{u^{2-4a}(1-u)^{2a}}. \label{OneIntHittingProb}
\end{align}
Letting
\begin{align*}
c_a = \frac{\Gamma(2a)}{\Gamma(1-2a) \Gamma(4a)},
\end{align*}
it is easy to see that
\begin{align*}
F(\epsilon) \sim c_a \epsilon^{4a-1}
\end{align*}
as $\epsilon \downarrow 0$. By the SLE scaling relations, it follows
that for $x > 0$
\begin{align*}
\prob{\gamma \cap [x, x+\epsilon] \neq \emptyset} = \prob{\gamma
\cap \left[ 1, 1 + \frac{\epsilon}{x} \right] \neq \emptyset} = F
\left( \frac{\epsilon}{x} \right) \sim c_a \left( \frac{\epsilon}{x}
\right)^{\beta}.
\end{align*}
Having observed the curve $\gamma[0,t]$ and mapping back to the
half-plane via $h_t$, it follows that, on the event $t < T_x$,
\begin{align}
\prob{\left. \gamma \cap [x,x+\epsilon] \neq \emptyset \right |
\gamma[0,t]} = F \left( \frac{h_t(x+\epsilon) - h_t(x)}{h_t(x)} \right). \label{condAsymptotics}
\end{align}
The map $h_t$ is also analytic at $x$ on the event $t < T_x$, and therefore
\begin{align*}
F \left( \frac{h_t(x+\epsilon) - h_t(x)}{h_t(x)} \right) \sim
c_a \left( \frac{\epsilon h_t'(x)}{h_t(x)} \right)^{\beta}, \epsilon
\downarrow 0.
\end{align*}
Thus
\begin{align}\label{MtAsCondProb}
M_t(x) = \lim_{\epsilon \downarrow 0} \epsilon^{-\beta} \prob{\left.
\gamma \cap [x,x+\epsilon] \neq \emptyset \right | \gamma[0,t]}/c_a.
\end{align}
In this limiting sense $M_t(x)$ describes the
conditional probability, having observed $\gamma[0,t]$, that the
curve will hit the point $x$. Consequently the following result is
not surprising, and merely expresses the fact that any fixed point
on the line is almost surely not hit by the curve, in the range $4 <
\kappa < 8$.

\begin{lemma}\label{MEndsAtZero}
Fix an $x > 0$. Then with probability 1,
\begin{align*}
\lim_{t \uparrow T_x} M_t(x) = 0.
\end{align*}
\end{lemma}

\begin{proof}
From \eqref{MDiff}, it follows that
\begin{align*}
d \log M_t(x) = -\frac{\beta}{h_t(x)} dB_t - \frac{1}{2} \left( \frac{\beta}{h_t(x)} \right)^2 dt
\end{align*}
for $t < T_x$, and therefore
\begin{align*}
M_t(x) = x^{-\beta} \exp \left \{ -\int_0^t \frac{\beta}{h_s(x)} dB_s
- \frac{1}{2} \int_0^t \left( \frac{\beta}{h_s(x)} \right)^2 ds
\right \}.
\end{align*}
Letting
\begin{align*}
u(t) = \int_0^t \left( \frac{\beta}{h_s(x)} \right)^2 ds.
\end{align*}
and considering the time changed martingale $\tilde{M}_u(x) =
M_{t^{-1}(u)}(x)$, we have
\begin{align*}
\tilde{M}_u(x) = x^{-\beta} \exp \left \{ \tilde{B}_u - \frac{u}{2} \right \},
\end{align*}
where $\tilde{B}_u$ is the standard Brownian motion defined by
$$
\tilde{B}_u = B_{t^{-1}(u)}.
$$
On this time scale it is clear that $\tilde{M}_u \to 0$ almost surely as $u
\uparrow \infty$ since $\tilde{B}_u - \frac{u}{2} \to -\infty$. If $u \left( T_x \right) = \infty$, then $$\lim_{t
\uparrow T_x} M_t(x) = \lim_{u \uparrow \infty} \tilde{M}_u(x) = 0$$
and the proof is complete. A number of authors, for example
\cite{dubedat:triangle} or \cite{lawler:book}, prove that $u \left(
T_x \right) = \infty$, but the proof is so simple that we repeat it
here. Let $S_0 = 0$, $S_n = \inf \left \{ t \geq 0 : h_t(x) = x2^{-n}
\right \}$. Then
\begin{align*}
\int_0^{T_x} \frac{ds}{h_s(x)^2} = \sum_{n=0}^{\infty}
\int_{S_n}^{S_{n+1}} \frac{ds}{h_s(x)^2}.
\end{align*}
The summands are independent by the Strong Markov Property, and
identically distributed by the scaling property of
$h_t(x)$ from Proposition \ref{BesselProp}(a). Since they are all positive it follows that $u \left( T_x
\right) = \infty$ almost surely.
\end{proof}

\begin{remark}
Schramm and Zhou also give a proof of Lemma \ref{MEndsAtZero}, although their proof uses an extremal length argument involving the SLE hull itself, whereas ours is purely probabilistic.
\end{remark}

\begin{corollary}\label{zeroLebesgue}
The set of exceptional $x > 0$ for which $\lim_{t \uparrow T_x} M_t(x) \neq 0$
has Lebesgue measure zero, almost surely.
\end{corollary}

\begin{proof}
This follows by Fubini's theorem, since the Lebesgue measure is
non-negative and
\begin{align*}
\expect{\left| \left \{ x > 0 : \lim_{t \uparrow T_x} M_t(x) \neq 0
\right \} \right|} = \int_0^{\infty} \prob{\lim_{t \uparrow T_x}
M_t(x) \neq 0} \, dx = 0.
\end{align*}
\end{proof}

\begin{remark}
In the introduction we showed that the natural supermartingale to consider in constructing the measure is $\int_{I \backslash K_t} M_t(x) \, dx$.
Corollary \ref{zeroLebesgue} shows that
\begin{align*}
\prob{\int_{I \cap K_t} M_t(x) \, dx = 0 \textnormal{ for all } t \geq 0} = 1.
\end{align*}
Hence we gain nothing by integrating over $I \cap K_t$ as well, so that
\begin{align*}
\prob{\int_{I \backslash K_t} M_t(x) \, dx = \int_I M_t(x) \, dx \textnormal{ for all } t \geq 0} = 1.
\end{align*}
In the rest of the paper we will write
\begin{align*}
X_t(I) := \int_I M_t(x) \, dx
\end{align*}
which is notationally more convenient. Equivalently one can also adopt the convention that $M_t(x) = 0$ for $t \geq T_x$, which also gives equality of the two integrals.
\end{remark}

It is important to note that $M_t(x)$ is a local martingale only,
and not a proper martingale as the next result shows. A similar result for the corresponding interior point martingale appears in \cite{lawler_sheff:time},  although they were not able to compute an explicit expression for $\expect{M_t(x)}$.

\begin{proposition}\label{localMartProp}
The local martingale $M_t(x)$ is a supermartingale that is not a
proper martingale. In fact
\begin{align*}
\expect{M_t(x)} = \Qxprob{T_x > t}x^{-\beta},
\end{align*}
where $Q_x$ is the measure on paths such that
\begin{align*}
dB_t = \frac{-\beta}{h_t(x)} \, dt + dW_t, \,\, t < T_x
\end{align*}
for some $Q_x$-Brownian motion $W_t$, and consequently
\begin{align*}
dh_t(x) = \frac{1-3a}{h_t(x)} \, dt + dW_t, \,\, t < T_x
\end{align*}
under $Q_x$. The probability $\Qxprob{T_x > t}$ can be explicitly computed as
\begin{align}
\Qxprob{T_x > t} = \frac{1}{\Gamma \left( 3a - 1/2 \right)} \int_0^{x^2/2t} u^{3a-3/2} e^{-u} du. \label{martExpect}
\end{align}
\end{proposition}

\begin{proof}
Since $M_t(x)$ is a positive local martingale it is automatically a
supermartingale. For $\te{M}(x)$, a corollary of \eqref{MDiff} is that
\begin{align}
\frac{d \te{M}(x)}{\te{M}(x)} = \frac{-\beta}{h_t(x)} \indicate{t \leq
\xe{T}} dB_t. \label{Girsanov}
\end{align}
Since $\te{M}(x)$ is a proper martingale, we can
define a new probability measure $Q_x$ by
\begin{align*}
\left. \frac{dQ_x}{dP} \right | _{\mathcal{F}_t} =
\frac{\te{M}(x)}{M_0^{\epsilon}(x)}.
\end{align*}
Note that $Q_x$, as defined above, also implicitly depends on $\epsilon$. Since
\begin{align}\label{MartTemp}
\te{M}(x) = M_t(x) \indicate{\xe{T} > t} + M_{\xe{T}}
\indicate{\xe{T} \leq t},
\end{align}
and $\te{M}(x)$ is a $P$-martingale, by taking expectations on both sides of \eqref{MartTemp},
\begin{align*}
\expect{M_t(x) \indicate{\xe{T} > t}} = M_0^{\epsilon}(x) -
\expect{M_{\xe{T}}(x) \indicate{\xe{T} \leq t}}.
\end{align*}
Note that all expectations are with respect to $P$. The event $\{ \xe{T} \leq t \}$ is
$\mathcal{F}_{\xe{T}}$-measurable, so therefore
\begin{align*}
\expect{M_{\xe{T}}(x) \indicate{\xe{T} \leq t}} &= M_0^{\epsilon}(x) \expect{\frac{M_{\xe{T}}(x)}{M_0^{\epsilon}(x)} \indicate{ \xe{T} \leq t}} \\
&= M_0^{\epsilon}(x) \expect{ \left. \frac{dQ_x}{dP} \right|_{\mathcal{F}_{\xe{T}}} \indicate{\xe{T} \leq t} } \\
&= M_0^{\epsilon}(x) \Qxprob{\xe{T} \leq t}.
\end{align*}
Since $M_t(x) = 0$ $P$-almost surely for $t \geq T_x$,
\begin{align*}
\expect{M_t(x)} &= \expect{M_t(x) \indicate{T_x > t}} \\
&= \lim_{\epsilon \downarrow 0} \expect{M_t(x) \indicate{\xe{T} > t}} \\
&= \lim_{\epsilon \downarrow 0} \left( 1 - \Qxprob{\xe{T} \leq t} \right) M_0^{\epsilon}(x) \\
&= \Qxprob{T_x > t} x^{-\beta}.
\end{align*}
The second equality is by monotone convergence. From this identity, we can prove that $M_t(x)$ is not a local martingale under $P$ by showing that $\Qxprob{T_x > t} < 1$ for $t$ sufficiently large. By \eqref{htdiff}, \eqref{Girsanov}, and Girsanov's Theorem, it follows that $h_t(x)$ satisfies the SDE
\begin{align*}
dh_t(x) = \frac{a - \beta}{h_t(x)} dt + dW_t = \frac{1-3a}{h_t(x)} dt + dW_t,
\end{align*}
under $Q_x$ (at least for $t < \xe{T}$), where $W_t$ is a $Q_x$-Brownian motion. But for $1/4 < a < 1/2$ we have that $-1/2 < 1-3a < 1/4$, and it is immediate from Proposition \ref{BesselProp}(c) that $\Qxprob{T_x > t} < 1$.

To compute $\Qxprob{T_x > t}$ exactly, we refer the reader to \cite[p. 98, Proposition 1]{yor_expfuncs}, where it is proved that, under $Q_x$, $x^2/(2T_x)$ has the gamma density with parameter $3a - 1/2$, i.e.
\begin{align*}
\Qxprob{ \frac{x^2}{2 T_x} \in dt } = \frac{t^{3a-3/2} e^{-t}}{\Gamma(3a-1/2)} dt.
\end{align*}
Note that in the notation of \cite{yor_expfuncs}, Yor's $T_0$ is the same as our $T_x$, his $a$ is our $x$, and his $\nu$ is the same as $3a-1/2$.
\end{proof}

\begin{remark}
The reference \cite{yor_expfuncs} is an excellent source for computations involving Bessel processes. Specifically, we point out that the explicit Radon-Nikodym derivatives between Bessel laws of different dimensions \cite[p. 97, formula (2.c)]{yor_expfuncs} can be used to compute $\expect{M_t(x)}$ in an alternative (but equivalent) way. These Radon-Nikodym derivatives appear in the SLE context in \cite{werner:hiding}.
\end{remark}

\begin{remark}
The case $a = 1/3 \, (\kappa = 6)$ is particularly interesting,
since in that case $h_t(x)$ is a simple Brownian motion under $Q_x$, and
\begin{align*}
\Qxprob{T_x > t} = \Qxprob{\min_{0 \leq s \leq t} B_s^* > 0},
\end{align*}
where $B_s^*$ is a $Q_x$-Brownian motion with $B_0^* = x$. By symmetry
this is the same as
\begin{align*}
\Qxprob{\max_{0 \leq s \leq t} B_s^{**} < x}
\end{align*}
where $B_s^{**}$ is a $Q_x$-Brownian motion with $B_0^{**} = 0$. The reflection principle for Brownian motion (see \cite{karatzas_shreve}) gives the well-known result that the running maximum of a Brownian motion has the same law as its absolute value, so that the latter is just
\begin{align*}
\Qxprob{ \left| B_t^{**} \right| < x} = \sqrt{ \frac{2}{\pi t} }
\int_0^x e^{-y^2/2t} \, dy = \sqrt{ \frac{2}{\pi} }
\int_0^{x/\sqrt{t}} e^{-y^2/2} \, dy.
\end{align*}
This agrees with the formula of Proposition \ref{localMartProp} by a simple change of variables.
\end{remark}

\subsection{The Schramm-Zhou Two-Point Martingale}

Schramm and Zhou were able to derive a so called two-point
martingale for $M$. Let
\begin{align*}
u(z) = (1-z)^{-\beta} \, _2F_1(2a, 1-4a, 4a; 1-z),
\end{align*}
where $_2F_1$ denotes the hypergeometric function; see \cite{AbSteg} for their properties. Let $0 < x < y$.
Then an application of Ito's Lemma shows that
\begin{align*}
u \left( \frac{h_t(x)}{h_t(y)} \right) M_t(x) M_t(y)
\end{align*}
is a local martingale for $t \in [0, T_x)$. Using properties of hypergeometric functions, Schramm and Zhou also show that
\begin{align*}
q_1 &:= \inf_{z \in (0,1)} u(z), \\
q_2 &:= \sup_{z \in (0,1)} (1-z)^{\beta}u(z)
\end{align*}
are both finite and positive. From this they are able to derive the
following two results (see Sections 2.2 and 4 of
\cite{schramm_zhou:dimension}), which we summarize here.

\begin{proposition}\label{ExpectationBound}
Let $0 < x <  y$, and  $\tau$ be a stopping time with $\prob{\tau <
\infty} = 1$. Then for $0 < \kappa < 8$ there exists a constant $c$,
depending only on $\kappa$, such that
\begin{align*}
\expect{M_{\tau}(x) M_{\tau}(y)} \leq c x^{-\beta} (y-x)^{-\beta}.
\end{align*}
\end{proposition}

\begin{proof}
Let $Z_t = h_t(x)/h_t(y)$. Recall that $u \left( Z_{t \wedge \tau}
\right) M_{t \wedge \tau}(x) M_{t \wedge \tau}(y)$ is a positive
local martingale, and therefore a supermartingale. By Fatou's Lemma,
\begin{align*}
\expect{M_{\tau}(x) M_{\tau}(y)} &\leq \liminf_{t \uparrow \infty} \expect{M_{t \wedge \tau}(x) M_{t \wedge \tau}(y)} \\
&\leq \liminf_{t \uparrow \infty} \expect{ u \left( Z_{t \wedge \tau } \right) M_{t \wedge \tau}(x) M_{t \wedge \tau}(y)}/q_1 \\
& \leq u(Z_0) M_0(x) M_0(y)/q_1 \\
& = u(x/y) x^{-\beta} y^{-\beta}/q_1 \\
& \leq \frac{q_2}{q_1}  (1-x/y)^{-\beta} x^{-\beta} y^{-\beta}.
\end{align*}
The third inequality uses the supermartingale property.
\end{proof}

\begin{corollary}\label{IndicatorBound} Let $0 < x < y$ and $\epsilon_x, \epsilon_y \geq
0$. Then
\begin{align*}
\prob{x \in C_t^{\epsilon_x}, y \in C_t^{\epsilon_y}} \leq c \left(
\epsilon_x \epsilon_y \right)^\beta x^{-\beta} \left( y - x
\right)^{-\beta}.
\end{align*}
for the same constant $c$ as in Proposition \ref{ExpectationBound}.
\end{corollary}

\begin{proof}
First note that $\epsilon^{-\beta} \indicate{x \in C_t^\epsilon}
\leq M_t^{\epsilon}(x)$. Therefore
\begin{align*}
\prob{x \in C_t^{\epsilon_x}, y \in C_t^{\epsilon_y}} \leq \left(
\epsilon_x \epsilon_y \right)^{\beta} \expect{M_t^{\epsilon_x}(x)
M_t^{\epsilon_y}(y)}.
\end{align*}
Proposition \ref{ExpectationBound} comes close to completing the
proof, except that $M_t^{\epsilon_x}(x)$ and $M_t^{\epsilon_y}(y)$
are stopped at different times ($t \wedge T_x^{\epsilon_x}$ for the
former, and $t \wedge T_y^{\epsilon_y}$ for the latter). Let $\tau =
t \wedge T_x^{\epsilon_x} \wedge T_y^{\epsilon_y}$. If $\tau = t$, then
$M_t^{\epsilon_x}(x) M_t^{\epsilon_y}(y) = M_{\tau}(x) M_{\tau}(y)$ and Proposition \ref{ExpectationBound} applies.
If $\tau \neq t$, then Schramm-Zhou write: ``If $\tau =
T_x^{\epsilon_x} < \infty$, then $M_t^{\epsilon_x}(x)$ is constant
in the range $t \in [T_x^{\epsilon_x}, T_y^{\epsilon_y})$, while
$M_t^{\epsilon_y}(y)$ is a martingale. The symmetric statement also
holds when we exchange $x$ and $y$." This implies that
\begin{align*}
\expect{M_t^{\epsilon_x}(x) M_t^{\epsilon_y}(y)} =
\expect{M_{\tau}(x) M_{\tau}(y)}.
\end{align*}
Proposition \ref{ExpectationBound} then finishes the proof.
\end{proof}

\subsection{The Sets $\te{C}$}

Here we describe the relation between the sets $\te{C}$ and $\hitpointst$. In light of \eqref{MtAsCondProb}, which roughly describes $M_t(x)$ as the conditional probability that $x$ is hit by the curve, we should expect that if $M_t(x)$ grows large then the curve must be close to $x$. The next lemma quantifies this intuition. It can also be found in both
\cite{alberts_sheff:dimension} and \cite{schramm_zhou:dimension}, and more detailed proofs can be found in those papers.

\begin{lemma}\label{ctDistance}
If $x \in \te{C}$ then $\operatorname{dist}(x, \gamma[0,t]) \leq 4
\epsilon$. Consequently, $\left | \te{C} \right |$ is an increasing
function of $t$ only at times $t$ for which the curve is within
distance $4 \epsilon$ of the real line.
\end{lemma}

\begin{proof}
Apply the Koebe 1/4 Theorem to the map $h_t$ extended across the
real axis by Schwarz reflection. It follows immediately that
\begin{align*}
\dist \left( x, \gamma[0,t] \right) \leq 4\frac{h_t(x)}{h_t'(x)},
\end{align*}
and then, since $\beta > 0$,
\begin{align*}
M_t(x) \leq \left( \frac{4}{\dist \left(x, \gamma[0,t] \right) }
\right)^{\beta}.
\end{align*}
The rest follows from the definition of $\te{C}$.
\end{proof}

\begin{remark}
Lemma \ref{ctDistance} shows that $C_t \subset \hitpointst$. It would be extremely helpful if the converse to the lemma were true, i.e. if $\dist(x, \gamma[0,t]) \leq K \epsilon$ for some constant $K$, then $x \in \te{C}$. If this were true then we would have $C_t = \hitpointst$. Schramm and Zhou are able to give a partial converse when the curve approaches the real line ``without making fjords''. In this paper we will not use their converse, but we note their important result that $\haussdim C_t = \haussdim \hitpointst$ for all $t > 0$ (they only prove the case $t = \infty$, but scaling properties easily extend the result to all $t > 0$). Hence, at least as measured by Hausdorff dimension, $C_t$ is not much smaller than $\hitpointst$, and in particular is non-empty for all $t > 0$.
\end{remark}

\begin{remark}
The sets $\te{C}$ are meant to act as a ``thickening'' of $\hitpointst$ by intervals whose length is of order $\epsilon$, but as the last remark shows, $\te{C}$ may miss still miss some points of $\hitpointst$. They also have the opposite problem: they may include too much. Consider the case where the SLE curve comes close to an interval without ever touching the real line nearby. The lemma says that the points of the interval will likely (but not necessarily) belong to $\te{C}$, even though they may be far from $\hitpointst$. For SLE curves this is not much of a problem, since transience of the curve means that it cannot come arbitrarily close to the interval without hitting it, and therefore when $\epsilon$ becomes sufficiently small the points in the interval will cease to belong to $\te{C}$. Note, however, that the $\epsilon$ at which these ``extra points'' vanish is a function of the curve in question, and therefore random.
\end{remark}

\section{The Doob-Meyer Decomposition \label{DMSection} }

\subsection{The Doob-Meyer Decomposition for Supermartingales}

In this section we briefly review some basic facts about the Doob-Meyer decomposition. We will only state the definitions and theorems that we will use, and refer the reader to \cite{karatzas_shreve} or \cite{revuz_yor} for very well written introductions to the Doob-Meyer theory. For an extremely detailed and rich treatment, we recommend \cite{dellacherie_meyer}. The notation we use will most closely resemble Section 1.4 of \cite{karatzas_shreve} and is self-contained to this section; specifically it does not refer to notation of previous or future sections unless explicitly stated so.

Throughout this section we will assume that $Z_t$ is a supermartingale with respect to a filtration $\Ft$, defined on some interval of time $[0, \zeta]$, for a stopping time $\zeta$ called the \textit{lifetime} of $Z$. As before, we assume that $\Ft$ satisfies the \textit{usual conditions}.

\begin{definition}
The supermartingale $Z$ is said to be of \textit{class $\mathcal{D}$} if the family
\begin{align*}
\{ Z_{\tau} : \tau \leq \zeta \textrm{ is an almost surely finite stopping time} \}
\end{align*}
is uniformly integrable.
\end{definition}

\begin{definition}
$Z$ is said to be \textit{regular} if for every $l > 0$, and every non-decreasing sequence of stopping times $\tau_n$ with $\prob{\tau_n \leq l} = 1$ and $\tau := \lim_{n \to \infty} \tau_n$, one has
\begin{align*}
\lim_{n \to \infty} \expect{Z_{\tau_n}} = \expect{Z_{\tau}}.
\end{align*}
\end{definition}

\begin{definition}
The \textit{predictable $\sigma$-field} is the coarsest $\sigma$-field on $\Omega \times \Rplus$ for which all continuous, $\Ft$-adapted processes are measurable. A process $A_t$ is predictable if the map $(\omega, t) \to A_t(w)$ from $\Omega \times \Rplus$ into $(\R_+, \borelplus)$ is measurable with respect to the predictable $\sigma$-field.
\end{definition}

\begin{theorem}[Doob-Meyer Decomposition]\label{DMTheorem}
Let $Z$ be a supermartingale of \textit{class $\mathcal{D}$} defined on $[0, \zeta]$. Then there exists a predictable, non-decreasing process $A$ that is right-continuous with left limits, such that $A_0 = 0$, $A_{\zeta}$ is integrable, and
\begin{align*}
Z_t = \condexpect{A_{\zeta} - A_t}{\F_t} + \condexpect{Z_{\zeta}}{\F_t}.
\end{align*}
If one defines $M_t := \condexpect{A_{\zeta} + Z_{\zeta}}{\F_t}$, then $M$ is a martingale and the above representation can be written as
\begin{align*}
Z_t = M_t - A_t.
\end{align*}
This decomposition is unique up to indistinguishability, i.e. if $M'$ and $A'$ are a martingale and a predictable, non-decreasing process (resp.) satisfying the above properties, then
\begin{align*}
\prob{M_t = M_t', \,\, A_t = A_t' \,\, \forall \, t \geq 0} = 1.
\end{align*}
Lastly, if $Z$ is regular then $A$ is continuous. If $Z$ is continuous, then $M$ and $A$ are both continuous.
\end{theorem}

Our use of the decomposition will be to show that a specific supermartingale $X$ (to be defined in the next section) is regular and of \textit{class $\mathcal{D}$}. This $X$ will satisfy $X_{\zeta} = 0$ almost surely, so that in particular $M_t = \condexpect{A_{\zeta}}{\F_t}$.

\begin{remark}
The notion of predictability is somewhat technical, but is best understood in the setting of discrete parameter processes where the analogous condition is that ``$A_n$ is $\F_{n-1}$-measurable''. We want to emphasize that predictability is important for the uniqueness part of the Doob-Meyer decomposition, but it is not a concept that we will explicitly use in this paper. The supermartingale $X_t(I)$ that we will decompose will turn out to be continuous, and therefore its non-decreasing part will be as well, and it suffices to know that continuous processes are always predictable. Moreover, predictable processes are always adapted to $\Ft$ (by definition), and so condition three of Theorem \ref{mainTheorem} implies that $\mum{\cdot \cap K_t}$ is $\Ft$-measurable. This further implies that the measure of an interval is completely determined at its swallowing time, as was previously mentioned.
\end{remark}

We will also consider a certain $\epsilon$-approximation $X^{\epsilon}$ to our supermartingale $X$, and the corresponding Doob-Meyer decomposition $X^{\epsilon} = M^{\epsilon} - A^{\epsilon}$. It will turn out that $X^{\epsilon}$ increases to $X$ as $\epsilon \downarrow 0$, so it is natural to ask if the corresponding parts of the Doob-Meyer decomposition might also converge. The next theorem gives an affirmative answer.

\begin{theorem}{\cite[Chapter VII, Section 20]{dellacherie_meyer}}\label{DellMTheorem}
Let $Z^n$ be an increasing sequence of positive supermartingales, where the limit $Z$ belongs to \textit{class $\mathcal{D}$} and is regular. Let $A^n$ and $A$ denote the non-decreasing processes associated to $Z^n$ and $Z$, respectively. Then for all stopping times $T$,
\begin{align*}
\lim_{n \to \infty} \expect{\left| A_T^{n} - A_T \right|} = 0.
\end{align*}
\end{theorem}

\subsection{The Doob-Meyer Decomposition Without the \textit{Class $\mathcal{D}$} Property}

One might naturally wonder if the \textit{class $\mathcal{D}$} property is actually important in the above decomposition. As we mentioned in the introduction, even without the \textit{class $\mathcal{D}$} property any supermartingale can always be uniquely decomposed as a local martingale minus a predictable, non-decreasing process. For a precise statement of this result we refer the reader to \cite[Chapter VII, Section 12]{dellacherie_meyer}. Any positive local martingale is automatically a supermartingale and therefore admits this decomposition, but by uniqueness the non-decreasing part must be zero. In particular, $M_t(x)$ has only this trivial decomposition. This tells us how \textit{not} to build our measure: the Doob-Meyer decomposition of the integral
\begin{align*}
\int_I M_t(x) \, dx
\end{align*}
is \textit{not} the integral of the Doob-Meyer decomposition for $M_t(x)$. In other words, the Doob-Meyer decomposition does not necessarily commute with integration.

\section{Construction of the Measure \label{ConstructSection} }

In this section we use the Doob-Meyer decomposition to construct the measure on $\hitpoints_+$ described in Theorem \ref{mainTheorem}. Most of the section is devoted to analyzing the process $X_t(I)$.

\subsection{The Process $X_t(I)$ \label{XtMeasure}}

In this section we study the stochastic process
\begin{align*}
X_{t}(I) := \int_I M_t(x) \, dx
\end{align*}
for intervals $I \subset \Rplus$. With this definition $X_t$ is the random measure on $\Rplus$ whose Radon-Nikodym derivative with respect to Lebesgue measure is $M_t(x)$. However, in this section we will mostly be concerned with the process $X_t(I)$ for a fixed interval $I$. Most of the results can be generalized from intervals to arbitrary $A \in \borelplus$.

First note $X_t(I)$ is finite almost surely at a fixed time $t \geq 0$ since, by Fubini's theorem,
\begin{align*}
\expect{X_t(I)} = \int_I \expect{M_t(x)} \, dx < \infty.
\end{align*}
The last inequality is an easy consequence of the explicit form of $\expect{M_t(x)}$ in \eqref{martExpect}. Fubini's theorem also shows that $X_t(I)$ is a supermartingale, since
\begin{align*}
\condexpect{X_t(A)}{\F_s} = \int_A \condexpect{M_t(x)}{\F_s} \, dx \leq \int_A M_s(x) \, dx = X_s(A).
\end{align*}
for $0 \leq s < t$. The inequality follows from $M_t(x)$ being a non-negative supermartingale. In fact $M_t(x)$ is a local martingale, which might lead one to speculate that $X_t(I)$ is also a local martingale. In the next section we will see that this is not the case. We will see that $X_t(I)$ admits a decomposition as a martingale minus a non-decreasing part that is not identically zero, and therefore it cannot be a local martingale.

We will also use the notation
\begin{align}
T_I := \inf \{ t \geq 0 : I \textrm{ is entirely swallowed by } \gamma \} = \sup_{x \in I} T_I,
\end{align}
which is again a stopping time. For $t \geq T_I$, Corollary \ref{zeroLebesgue} shows $M_t(x)$ as a function of $x$ is identically zero except for a possible set of Lebesgue measure zero. Hence $X_t(I) = 0$ for $t \geq T_I$.

Using Proposition \ref{ExpectationBound}, we are also able to give an upper bound on the expected squared mass of an interval $I$ under $X_t$.

\begin{proposition}\label{IntegralBound}
Let $I = (x_1, x_2] \subset \Rplus$ be an interval. Then there exists a constant $c_*$, depending only on $\kappa$, such that
\begin{align*}
\expect{X_{\tau}(I)^2} = \expect{\left( \int_I M_{\tau}(x) \, dx \right)^2} \leq c_* x_1^{-\beta} \left| I
\right|^{1+d}
\end{align*}
for every stopping time $\tau$ with $\prob{\tau < \infty} = 1$.
\end{proposition}

\begin{proof}
By Proposition \ref{ExpectationBound},
\begin{align*}
\expect{\left( \int_I M_{\tau}(x) \, dx \right)^2} &= \int_I \int_I
\expect{M_{\tau}(x) M_{\tau}(y)} \, dx \, dy \\
&\leq 2c \int_{x_1}^{x_2} \int_{x_1}^y x^{-\beta} (y-x)^{-\beta} \,
dx \, dy \\
& \leq 2c x_1^{-\beta} \int_{x_1}^{x_2} \int_{x_1}^y
\left( y-x \right)^{-\beta} \, dx \, dy \\
& = 2c \frac{x_1^{-\beta}}{1-\beta} \int_{x_1}^{x_2} (y-x_1)^{1-\beta} \, dy\\
& = 2c \frac{x_1^{-\beta}}{(1-\beta)(2-\beta)} (x_2 -
x_1)^{2-\beta} .
\end{align*}
Recalling that $d = 1 - \beta$ finishes the proof.
\end{proof}

\begin{corollary}\label{XtHolder}
Let $I = (x_1, x_2]$ with $0 < x_1 < x_2 < \infty$. Then as a function of $x_2$, $X_t(I)$ has a version that is H\"{o}lder-$\gamma$ continuous, for any $\gamma < d/2$.
\end{corollary}

\begin{proof}
This follows by an application of Proposition \ref{IntegralBound} and the Kolmogorov-Centsov Theorem.
\end{proof}

Using Proposition \ref{IntegralBound}, we are able to show that $X_t$ does not assign large amounts of mass to small intervals.

\begin{corollary}\label{SmallMass}
Let $I = (x_1, x_2]$ with $0 < x_1 < x_2 < \infty$. Let $\{ I_{k,n} \}_{1 \leq k \leq 2^n}$ be a partition of $I$ into $2^n$ subintervals of length $|I| 2^{-n}$. Then for $\alpha$ with $0 < \alpha < d/2$,
\begin{align*}
\prob{\max_{1 \leq k \leq 2^n} \sup_{t \geq 0} X_t \left( I_{k, n}
\right) \geq 2^{-n\alpha} \textnormal{ for infinitely many } n} = 0.
\end{align*}
\end{corollary}

\begin{proof}
For a fixed $\epsilon > 0$, recall that the process $\te{M}(x)$ is, for each $x > 0$, a postive \textit{martingale} that is bounded above by $\epsilon^{-\beta}$. By Fubini's Theorem, it follows that for an interval $I \subset \Rplus$ the process
\begin{align*}
\int_I \te{M}(x) \, dx
\end{align*}
is also positive martingale that is bounded above by $\epsilon^{-\beta} \leb{I}$. The square of the latter process is therefore a submartingale. Using that $M_{\infty}(x) = \epsilon^{-\beta} \indicate{x \in \ep{C}}$ and the bound
\begin{align*}
\expect{M^{\epsilon}_{\infty}(x) M^{\epsilon}_{\infty}(y)} = \epsilon^{-2\beta} \prob{x,y \in \ep{C}} \leq c x^{-\beta} (y-x)^{-\beta}
\end{align*}
of Corollary \ref{IndicatorBound}, an application of Doob's maximal inequality gives that
\begin{align*}
\prob{\sup_{t \geq 0} \int_I \te{M}(x) \, dx > \lambda} \leq \lambda^{-2} \int_I \int_I \expect{\ep{M_{\infty}}(x) \ep{M_{\infty}}(y)} \, dx \, dy \leq c_* x_1^{-\beta} \lambda^{-2} \leb{I}^{1+d},
\end{align*}
where $c_*$ is as in Proposition \ref{IntegralBound}. We will now extend the same bound to the integral of $M_t(x)$. First, recalling that
\begin{align*}
\prob{\lim_{\epsilon \downarrow 0} \xe{T} = T_x \,\, \forall \, x > 0 } = 1,
\end{align*}
and that $\te{M}(x) = M_{t \wedge \xe{T}}(x)$, we have
\begin{align}\label{smallMass2}
\prob{\liminf_{\epsilon \downarrow 0} \te{M}(x) \geq M_t(x) \,\, \forall \, x > 0, t \geq 0 } = 1.
\end{align}
In fact, we actually have that
\begin{align*}
\liminf_{\epsilon \downarrow 0} \te{M}(x) = \lim_{\epsilon \downarrow 0} \te{M}(x) = M_t(x)
\end{align*}
so long as $x \not \in C^{\infty}$. The only problem occurs when $x \in C_{\infty}$ and $t \geq T_x$, in which case $M_t(x)$ is defined to be zero, whereas $\liminf_{\epsilon \downarrow 0} \te{M}(x) = \infty$. Regardless, \eqref{smallMass2} is all that we require, and using it and Fatou's Lemma we have,
\begin{align*}
\prob{\int_I M_t(x) \, dx \leq \liminf_{\epsilon \downarrow 0} \int_I \te{M}(x) \, dx \,\, \forall \, t \geq 0} = 1.
\end{align*}
It is easy to verify the deterministic fact that
\begin{align*}
\sup_{t \geq 0} \liminf_{\epsilon \downarrow 0} \int_I \te{M}(x) \, dx \leq \liminf_{\epsilon \downarrow 0} \sup_{t \geq 0} \int_I \te{M}(x) \, dx,
\end{align*}
so we have that
\begin{align*}
\prob{\sup_{t \geq 0} \int_I M_t(x) \, dx \leq \liminf_{\epsilon \downarrow 0} \sup_{t \geq 0} \int_I \te{M}(x) \, dx } = 1.
\end{align*}
Therefore, on this last event of full probability,
\begin{align*}
\indicate{\sup_{t \geq 0} \int_I M_t(x) \, dx > \lambda} \leq \liminf_{\epsilon \downarrow 0} \indicate{\sup_{t \geq 0} \int_I \te{M}(x) \, dx > \lambda}.
\end{align*}
Consequently, via another application of Fatou's Lemma,
\begin{align}
\prob{\sup_{t \geq 0} \int_I M_t(x) \, dx > \lambda} &\leq \expect{\liminf_{\epsilon \downarrow 0} \indicate{\sup_{t \geq 0} \int_I \te{M}(x) \, dx > \lambda} } \notag \\
&\leq \liminf_{\epsilon \downarrow 0} \prob{\sup_{t \geq 0} \int_I \te{M}(x) \, dx > \lambda} \notag \\
&\leq c_* x_1^{-\beta} \lambda^{-2} \leb{I}^{1+d}. \label{smallMassBound}
\end{align}

From \eqref{smallMassBound} the proof of the lemma follows easily. We have
\begin{align*}
\prob{\max_{1 \leq k \leq 2^n} \sup_{t \geq 0} X_t(I_{k,n}) \, dx \geq 2^{-n\alpha}} &= \prob{\bigcup_{k=1}^{2^n} \left \{ \sup_{t \geq 0} X_t(I_{k,n}) \geq 2^{-n \alpha} \right \} } \\
&\leq \sum_{k=1}^{2^n} \prob{\sup_{t \geq 0} X_t(I_{k,n}) \geq 2^{-n \alpha} } \\
&\leq \sum_{k=1}^{2^n} 2^{2n \alpha} c_* x_1^{-\beta} \left(  \left| I \right| 2^{-n} \right)^{1+d} \\
&= c_* x_1^{-\beta} \left| I \right|^{1+d} 2^{-n(d - 2 \alpha)},
\end{align*}
and then an application of the Borel-Cantelli lemma completes the proof.

\end{proof}

We will also be interested in how the process $X_t(I)$ evolves over time. It is clear intuitively what is happening. When the tip of the SLE curve is not in $I$, the local martingales $M_t(x)$ do not grow large and therefore behave like martingales. Since $X_t(I)$ is an integral of these martingales (which are all positive), it follows that $X_t(I)$ also behaves like a martingale when the tip is not on $I$. This will force that the non-decreasing part of the Doob-Meyer decomposition for $X_t(I)$ can only be growing when $\gamma(t) \in I$, which is to be expected. The rest of this section puts this intuition on a solid foundation.

\begin{definition}
For an interval $I \subset \Rplus$, define the process $Y_t(I), t \geq 0$, by
\begin{align*}
Y_t(I) := \inf_{x \in I \backslash K_t} h_t(x).
\end{align*}
\end{definition}

Hence $Y_t(I)$ is keeping track of the leftmost unswallowed point of $I$ under the flow. It follows that $Y_t(I)$ is almost surely continuous in $t$, and that it goes to zero almost surely as $t \uparrow T_I$. The times $t < T_I$ for which $Y_t(I) = 0$ are exactly the times at which $\gamma(t) \in \overline{I}$ (note we have to take the closure since $I$ may not contain the endpoints of the interval). Using this process we define the following stopping times:

\begin{definition}
Fix an interval $I \subset \R_+$. For $\epsilon > 0$, recursively define the sequence of stopping times $\ep{\tau_n}$ by $\ep{\tau_0} = 0$ and
\begin{align*}
\ep{\tau_{2n+1}} &:= \inf \bigl \{ t > \ep{\tau_{2n}} : Y_t(I) \leq \epsilon \bigr \}, \\
\ep{\tau_{2n+2}} &:= \inf \bigl \{ t > \ep{\tau_{2n+1}} : Y_t(I) > 2\epsilon \bigr \}.
\end{align*}
\end{definition}

Hence intervals of time $\left( \ep{\tau_{2n}}, \ep{\tau_{2n+1}} \right)$ are downcrossings of the $Y_t(I)$ process from $2 \epsilon$ to $\epsilon$, and intervals $\left( \ep{\tau_{2n+1}}, \ep{\tau_{2n+2}} \right)$ are upcrossings from $\epsilon$ back to $2 \epsilon$. Therefore only upcrossings can contain the times at which $\gamma(t) \in I$, and on the downcrossings the set $I \cap K_t$ of swallowed points in $I$ is not growing. In other words, if $\ep{\tau_{2n}} < t < \ep{\tau_{2n+1}}$ then
\begin{align*}
I \backslash K_t = I \backslash K_{\ep{\tau_{2n}}}.
\end{align*}
During a downcrossing the $h_t(x)$ are uniformly bounded below by $\epsilon$, hence
\begin{align*}
\sup_{\ep{\tau_{2n}} < t < \ep{\tau_{2n+1}}} \sup_{x \in I \backslash K_t} M_t(x) \leq \epsilon^{-\beta},
\end{align*}
by inequality \eqref{MtIneq}. We also have the following property of the stopping times:

\begin{lemma}\label{tauToInfty}
For a fixed $\epsilon > 0$, $\ep{\tau_n} \to \infty$ almost surely as $n \to \infty$.
\end{lemma}

\begin{proof}
The $Y_t(I)$ process can only have finitely many upcrossings or downcrossings since it is continuous, and since the process goes to zero almost surely there is a last $n$ such that $\ep{\tau_{2n+1}} \leq T_I$. After that $\ep{\tau_m} = \infty$ for $m > 2n+1$.
\end{proof}

\begin{lemma}\label{contAway}
For a fixed $\epsilon > 0$ and for any $n \geq 0$, the process $\ep{Z_n}(t) := X_{t \wedge \ep{\tau_{2n+1}}}(I) - X_{t \wedge \ep{\tau_{2n}}}(I)$ is a continuous martingale.
\end{lemma}

\begin{proof}
First suppose that $N$ is the largest integer such that $\ep{\tau_{2N+1}} \leq T_I$. Then for $M > N$ we have $\ep{\tau_{2M}} = \ep{\tau_{2M+1}} = \infty$ and therefore $\ep{Z_M}(t)$ is identically zero.

For $n \leq N$, observe that we may write
\begin{align*}
X_{t \wedge \ep{\tau_{2n+1}}}(I) - X_{t \wedge \ep{\tau_{2n}}}(I) &= \int_I M_{t \wedge \ep{\tau_{2n+1}}}(x) - M_{t \wedge \ep{\tau_{2n}}}(x) \, dx \\
&= \int_{I \backslash K_{\ep{\tau_{2n}}}} M_{t \wedge \ep{\tau_{2n+1}}}(x) - M_{t \wedge \ep{\tau_{2n}}}(x) \, dx.
\end{align*}
For the integrand we have the bound
\begin{align*}
\sup_{t > 0} \sup_{x \in I \backslash K_{\ep{\tau_{2n}}}} \leb{M_{t \wedge \ep{\tau_{2n+1}}}(x) - M_{t \wedge \ep{\tau_{2n}}}(x)} \leq 2 \epsilon^{-\beta}
\end{align*}
(the bound is for all $t$ since the process is clearly flat for $t \not \in (\ep{\tau_{2n}}, \ep{\tau_{2n+1}})$). Hence for each $x \in I \backslash K_{\ep{\tau_{2n}}}$ the integrand is a local martingale that is uniformly bounded above, and therefore is a martingale. An application of Fubini's Theorem shows that $\ep{Z_n}(t)$ is also a martingale. To get the continuity, observe that $\ep{Z_n}(t)$ is a martingale with respect to the Brownian filtration $\Ft$, and therefore by the martingale representation theorem (see, for example, \cite[Chapter IV, Theorem 43 and Corollary 1]{protter}) is automatically continuous.
\end{proof}

\begin{corollary}\label{contToFirstHitting}
Fix an interval $I \subset \Rplus$. Let $\tau_I = \inf \{ t \geq 0 : \gamma(t) \in \overline{I} \}$. Then the process $X_{t \wedge \tau_I}(I)$ is a continuous martingale.
\end{corollary}

\begin{proof}
Let $\ep{\tau_I} := \inf \{ t \geq 0 : \dist(I, \gamma[0,t]) \leq 4 \epsilon \}$. Then clearly $\ep{\tau_I}$ increases to $\tau_I$ almost surely as $\epsilon \downarrow 0$. Let
\begin{align*}
\ep{T_I} := \inf \left \{ t \geq 0 : \sup_{x \in I} M_t(x) \geq \epsilon^{-\beta} \right \} = \inf_{x \in I} \ep{T_x}.
\end{align*}
If $\ep{T_I} < \infty$, then since $M_t(x)$ is analytic in $x$ there must be an $x_0 \in \overline{I}$ such that $M_{\ep{T_I}}(x_0) \geq \epsilon^{-\beta}$. Lemma \ref{ctDistance} then implies that $\dist(I, \gamma[0, \ep{T_I}]) \leq 4 \epsilon$, hence $\ep{\tau_I} \leq \ep{T_I}$. We will actually show that the process $X_{t \wedge \ep{T_I}}(I)$ is a continuous martingale for every $\epsilon > 0$, and then take a limit as $\epsilon \downarrow 0$ to prove the statement. Note that by definition of $\ep{T_I}$,
\begin{align*}
\sup_{t \geq 0} \sup_{\epsilon > 0} M_{t \wedge \ep{T_I}}(x) \leq \epsilon^{-\beta}.
\end{align*}
Hence the $M_{t \wedge \ep{T_I}}(x)$ are local martingales that are uniformly bounded above and so are martingales. The standard application of Fubini's theorem then shows that
\begin{align*}
X_{t \wedge \ep{T_I}}(I) = \int_I M_{t \wedge \ep{T_I}}(x) \, dx
\end{align*}
is also a martingale. Continuity of $X_{t \wedge \ep{T_I}}$ follows again by the martingale representation theorem.
\end{proof}

\begin{remark}
Note that $\tau_I$ is very closely related to the first time that $Y_t(I) = 0$. In fact, on the event $\{ \tau_I < \infty \}$ they are the same, but on the event that the curve doesn't hit $\overline{I}$ the first time that $Y_t(I) = 0$ is the swallowing time $T_I$ of $I$, which is finite, while $\tau_I = \infty$.
\end{remark}

\begin{proposition}\label{contNotInI}
Fix an interval $I \subset \Rplus$. Then
\begin{align*}
\prob{X_t(I) \textrm{ is continuous on } \left \{ t \geq 0 : \gamma(t) \not \in \overline{I} \right \}} = 1.
\end{align*}
\end{proposition}

\begin{proof}
First suppose that $I$ is closed. Then $\gamma(t) \in I$ if and only if $Y_t(I) = 0$. Moreover, since $\gamma$ is almost surely continuous and $I$ is closed it follows that
\begin{align*}
\left \{ t \geq 0 : \gamma(t) \not \in I \right \}
\end{align*}
is almost surely open. From these two facts it follows that
\begin{align*}
\bigcup_{\substack{\epsilon > 0 \\ \epsilon \in \mathbb{Q}}} \bigcup_{n \geq 0} \left( \ep{\tau_{2n}}, \ep{\tau_{2n+1}} \right) = \left \{ t \geq 0 : \gamma(t) \not \in I \right \}.
\end{align*}
But by Lemma \ref{contAway} we have
\begin{align*}
\prob{X_t(I) \textrm{ is continuous on } \bigcup_{\substack{\epsilon > 0 \\ \epsilon \in \mathbb{Q}}} \bigcup_{n \geq 0} \left( \ep{\tau_{2n}}, \ep{\tau_{2n+1}} \right) } = 1.
\end{align*}

If $I$ is not closed then the statement is true for $X_t(\overline{I})$. But the processes $X_t(I)$ and $X_t(\overline{I})$ are indistinguishable, i.e.
\begin{align*}
\prob{X_t(I) = X_t(\overline{I}) \textrm{ for all } t \geq 0} = 1,
\end{align*}
since $\overline{I} \backslash I$ consists of at most the two endpoints of $I$.
\end{proof}

\begin{proposition}\label{XContinuous}
For any interval $I \subset \Rplus$, the process $X_t(I)$ is almost surely continuous with $X_t(I) = 0$ for $t \geq T_I$.
\end{proposition}

\begin{proof}
By Proposition \ref{contNotInI}, any discontinuity of $I$ could only occur at a time $t$ for which for which $\gamma(t) \in I$. We will show that Corollary \ref{SmallMass} forbids having a discontinuity at such times.

Suppose that $\gamma(T) \in I$. Let $\{ I_{k,n} \}_{1 \leq k \leq 2^n}$ be a partition of the interval $I$ into $2^n$ subintervals of length $|I| 2^{-n}$ (they may overlap at the endpoints). We divide the subintervals into four different types:
\begin{enumerate}
\renewcommand{\labelenumi}{(\roman{enumi})}
\item subintervals that have been completely swallowed \textit{strictly before} time $T$,
\item subintervals that have not been hit or swallowed before time $T$,
\item subintervals containing $\gamma(T)$,
\item subintervals containing $I \cap K_{T-}$.
\end{enumerate}
The $X_t$ process for subintervals of type (i) is identically zero after the swallowing time; hence $X_t(I_{k,n})$ is continuous at time $T$ if $I_{k,n}$ is of type (i). Corollary \ref{contToFirstHitting} implies that $X_t(I_{k,n})$ is continuous at time $T$ for subintervals of type (ii). Writing
\begin{align*}
X_t(I) = \sum_{k=1}^{2^n} X_t(I_{k,n}),
\end{align*}
it follows that a discontinuity of $X_t(I)$ at time $T$ can only be caused by intervals of type (iii) or (iv). There are at most two intervals of type (iii) (there is usually only one, there are two only if $\gamma(T)$ lies on a shared endpoint of different $I_{k,n}$). The type (iv) intervals are those containing the last point of $I$ that $\gamma$ hits before time $T$; as such there are at most two subintervals of type (iv). If $X_t(I)$ has a discontinuity of size $\delta$ at time $T$ then one of the four subintervals of type (iii) or (iv) must have a discontinuity of size $\delta/4$ at time $T$. Thus for all $n \geq 0$ there exists an integer $k_n$ such that
\begin{align*}
\sup_{|t-T| \leq \epsilon} X_t(I_{k_n,n}) - \inf_{|t-T| \leq \epsilon} X_t(I_{k_n,n}) \geq \delta/4
\end{align*}
for all $\epsilon > 0$. The infimum being non-negative implies that
\begin{align*}
\sup_{t > 0} X_t(I_{k_n,n}) \geq \delta/4.
\end{align*}
for all $n \geq 0$. But Corollary \ref{SmallMass} says that the latter event has probability zero, from which the result follows.
\end{proof}

\subsection{The Doob-Meyer Decomposition for $X_t(I)$}

For an interval $I = (x_1, x_2]$ with $0 < x_1 < x_2 < \infty$, we prove in this section that $X_t(I)$ has a Doob-Meyer decomposition as a martingale minus a predictable, non-decreasing process. The strategy is to use the proofs of Section \ref{XtMeasure} to verify that $X_t(I)$ satisfies the hypotheses of Theorem \ref{DMTheorem}.

\begin{proposition}\label{DMDecomp}
The process $X_t(I)$ can be \textit{uniquely} decomposed as
\begin{align*}
X_t(I) = N_t(I) - A_t(I),
\end{align*}
where $N_t(I)$ is a continuous martingale and $A_t(I)$ is a \textbf{continuous}, non-decreasing process with $A_0(I) = 0$ and $A_{T_I}(I)$ integrable. In fact, $N_t(I) = \condexpect{A_{T_I}(I)}{\F_t}$. Moreover, both $N_t(I)$ and $A_t(I)$ are constant for $t \geq T_I$.
\end{proposition}

\begin{proof}
By Theorem \ref{DMTheorem}, it is enough for the existence part of the decomposition to show that $X_t(I)$ is of \textit{class $\mathcal{D}$}. This property is immediate from Proposition \ref{IntegralBound}, since if $\tau$ is a stopping time with $\prob{\tau < \infty} = 1$ then
\begin{align*}
\expect{X_{\tau}(I)^2} < \infty,
\end{align*}
and the bound is independent of $\tau$. The continuity of $A_t(I)$ follows from the continuity of $X_t(I)$ in Proposition \ref{XContinuous}. In fact the continuity of $X_t(I)$ implies that both parts $N_t(I)$ and $A_t(I)$ are continuous.

To show that $N_t(I)$ and $A_t(I)$ are constant for $t \geq T_I$ it is enough to observe that $X_t(I) = 0$ for $t \geq T_I$. The martingale $N_t(I)$ can be equal to the non-decreasing process $A_t(I)$ only if the two processes are the same constant.
\end{proof}

\subsection{Definition of the Boundary Measure \label{defnSection}}

We now have all the tools we need to properly define the boundary measure. The basic construction is to take the terminal values $A_{T_I}(I)$ of the non-decreasing processes and encode them all into a single measure. We conclude the section with an alternative but useful characterization of the boundary measure.

\begin{definition}[Definition of the Boundary Measure]
Define the collection of intervals
\begin{align*}
\mathcal{Q} := \Bigl \{ I = (x_1, x_2] : 0 < x_1 < x_2 < \infty, x_1, x_2 \in \mathbb{Q} \Bigr \}.
\end{align*}
For each $I \in \mathcal{Q}$, define
\begin{align*}
\mu(I) := A_{T_I}(I).
\end{align*}
It is easy to see that $\mu$ is almost surely countably additive on the field generated by $\mathcal{Q}$ and that this field generates $\borelplus$, so by the Carath\'{e}odory Extension Theorem $\mu$ can be uniquely extended to a Borel measure on $\Rplus$.
\end{definition}

\begin{proposition}\label{muRestricted}
For all $t \geq 0$, the random measure $\mum{\cdot \cap K_t}$ is $\Ft$-measurable.
\end{proposition}

\begin{proof}
It suffices to prove that the random variable $\mum{I \cap K_t}$ is $\Ft$-measurable for each interval $I \subset \Rplus$. The set $I \cap K_t = \{ x \in I : T_x \leq t \}$ is $\Ft$-measurable since the $T_x$ are all stopping times. If $I \cap K_t = \emptyset$ then $\mum{I \cap K_t} = 0$. If $I \subset K_t$ then $t \geq T_I$, and by definition $\mum{I} = A_{T_I}(I)$.
\end{proof}

Hence for intervals $I \subset \Rplus$, $\mum{I \cap K_t}$ is a non-decreasing, adapted process such that
\begin{align*}
\mum{I \cap K_t} + \int_I M_t(x) \, dx
\end{align*}
is a martingale with respect to $P$ and $\Ft$. Uniqueness of the Doob-Meyer decomposition tells us that $\mum{I \cap K_t}$ is the unique (up to indistinguishability) such process that can be added to $\int_I M_t(x) \, dx$ to get a martingale. As in the introduction we therefore have the decomposition
\begin{align*}
\condexpect{\mu(I)}{\Ft} = \mum{I \cap K_t} + \int_I M_t(x) \, dx.
\end{align*}
The next theorem gives a similar but more precise characterization of the random measure $\mu$.

\begin{theorem}\label{decompTheorem}
Let $(\Omega, \F, P)$ be a probability space with a $P$-Brownian motion $\{B_t, \Ft; t \geq 0 \}$ on it. If the filtration $\Ft$ satisfies the \textit{usual conditions}, then there exists a unique random measure $\mu$ (unique up to an event of $P$-measure zero) such that
\begin{enumerate}
\roman{enumi}
\item $\mum{\cdot \cap K_t}$ is a predictable process,
\item $\mum{I \cap K_t} + \int_I M_t(x) \, dx$ is, for every interval $I \subset \Rplus$, a martingale with respect to $P$ and $\Ft$.
\end{enumerate}
\end{theorem}

Note that the random measure $\mu$ is implicitly a function of the Brownian motion $B_t$, its filtration $\Ft$, and the measure $P$.

\begin{remark}
To avoid having to constantly say ``for every interval $I \subset \Rplus$'', we will use the infinitesimal shorthand to state part two of the theorem as
\begin{align*}
\indicate{x \in K_t} \, d\mu(x) + M_t(x) \, dx
\end{align*}
is a martingale with respect to $P$ and $\Ft$.
\end{remark}

\begin{corollary}\label{decompLemma}
For all bounded, measurable functions $f : \Rplus \to \R$
\begin{align*}
\indicate{x \in K_t} f(x) \, d\mu(x) + f(x) M_t(x) \, dx
\end{align*}
is a martingale with respect to $P$ and $\Ft$. Moreover, if $f$ is non-negative then $f(x) d\mu(x)$ is the unique random measure whose restriction to $K_t$ is predictable and such that
\begin{align*}
\indicate{x \in K_t} f(x) \, d\mu(x) + f(x) M_t(x) \, dx
\end{align*}
is a martingale with respect to $P$ and $\Ft$.
\end{corollary}

\begin{proof}
The case $f(x) = 1$ is exactly statement two of Theorem \ref{decompTheorem}. Hence the lemma also holds if $f$ is a simple function (i.e. constant on intervals). The first statement is then proved by approximating bounded, measurable functions by simple functions.

For the second statement, note that $f \geq 0$ implies that $\int_I f(x) M_t(x) \, dx$ is a supermartingale (since $M_t(x)$ is). Uniqueness of the Doob-Meyer decomposition means that there is only one predictable non-decreasing process that can be added to $f(x) M_t(x) \, dx$ to get a martingale; by the first part that process must be $\indicate{x \in K_t} f(x) \, d\mu(x)$.
\end{proof}

\section{Properties of the Measure \label{muProperties}}

Having now constructed the random measure on $\Rplus$ that was described in the introduction, we proceed to show that it has all of the properties of Theorem \ref{mainTheorem}. We begin with the Domain Markov property.

\subsection{Domain Markov Property of $\mu$}

In this section we prove that the measure $\mu$ of Section \ref{defnSection} satisfies the Domain Markov property of Theorem \ref{mainTheorem}. The idea of the proof is intuitively clear. Given $\Ft$, consider the future SLE curve and hull mapped back to $\H$ via $h_t$, i.e.
\begin{align*}
\gamma^t(s) := h_t(\gamma(t+s)), \quad K_{t,s} := h_t(K_{t+s}), \,\, s \geq 0.
\end{align*}
Then $\gamma^t$ is independent of $\Ft$ but has the law of $\gamma$. Consequently the boundary measure corresponding to $\gamma^t$ is independent of $\Ft$ and has the law of the boundary measure for $\H$. We will show that the $d$-dimensional covariant transform of the boundary measure for $\gamma$ (restricted to $\Rplus \backslash K_t$) is exactly the boundary measure for $\gamma^t$ (restricted to $h_t(\Rplus \backslash K_t)$).

It is easier to prove the above using the Brownian motions that generate $\gamma$ and $\gamma^t$, rather than the curves themselves. We will also prove the more general version in which the fixed time $t$ is replaced by a stopping time $T$. If $\{B_t, \Ft; t \geq 0 \}$ is the Brownian motion generating $\gamma$ then clearly $\{B_{T,s}, \F_{T,s}; s \geq 0 \}$ is the Brownian motion generating $\gamma^T$. It is easy to verify that $\{B_{T,s}, \F_{T+s}; s \geq 0 \}$ is also a Brownian motion under $P$; we prove a theorem about it first. Note that $B_{T,s}$ generates both the sequence of conformal maps $h_{T,s}$ and the family of martingales $M_{T,s}(x)$.

\begin{theorem}\label{futureDecompTheorem}
Let $T$ be an $\Ft$-measurable stopping time. Let $\mu$ (respectively $\mu^*$) be the unique random measure of Theorem \ref{decompTheorem} associated to $\{B_t, \Ft; t \geq 0 \}$ (respectively $\{B_{T,s}, \F_{T+s}; s \geq 0 \}$). Then
\begin{align*}
\prob{\dtrans{h_T}{\mu} = \mu^* \textnormal{ restricted to } h_T(\R_+ \backslash K_T)} = 1.
\end{align*}
Further, $\mu^*$ is also the unique random measure associated to $\{B_{T,s}, \F_{T,s}; s \geq 0 \}$.
\end{theorem}

\begin{proof}
Since $h_T'(x)^d$ is a positive, continuous function, Corollary \ref{decompLemma} implies that $h_T'(x)^d d\mu(x)$ is the unique random measure on $\Rplus$ such that $\indicate{x \in K_{T+s}} h_T'(x)^d d\mu(x)$ is $\F_{T+s}$ measurable and
\begin{align}\label{measureMart}
\indicate{x \in K_{T+s}}h_T'(x)^d d\mu(x) + M_{T+s}(x) h_T'(x)^d \, dx
\end{align}
is a martingale (in $s$) with respect to $P$ and $\F_{T+s}$. By definition of the $d$-dimensional covariant transform we have $d \dtrans{h_T}{\mu}(h_T(x)) = h_T'(x)^d d\mu(x)$, and by the identity \eqref{mCommute} we have $M_{T+s}(x)h_T'(x)^d = M_{T,s}(h_T(x))h_T'(x)$. Making the change of variables $y = h_T(x)$, which is valid for $x \in \R_+ \backslash K_T$, equation \eqref{measureMart} therefore says that
\begin{align*}
\indicate{y \in K_{T,s}} d \dtrans{h_T}{\mu}(y) + M_{T,s}(y) \, dy
\end{align*}
is a martingale with respect to $P$ and $\F_{T+s}$. On the other hand, by definition $\mu^*$ is the unique random measure on $\Rplus$ such that
\begin{align*}
\indicate{y \in K_{T,s}} d\mu^*(y) + M_{T,s}(y) \, dy
\end{align*}
is a martingale with respect to $P$ and $\F_{T+s}$. Uniqueness forces that
\begin{align*}
d \dtrans{h_T}{\mu}(y) = d\mu^*(y).
\end{align*}
Note this equality only holds for $y \in h_T(\Rplus \backslash K_T)$, which explains why $\mu^*$ must be restricted to $h_T(\R_+ \backslash K_T)$ in the statement of the theorem.

Finally, note that both the measure $\mu^*(\cdot \cap K_{T,s})$ and the martingales $M_{T,s}$ are $\F_{T,s}$-measurable (they are all determined by $B_{T,s}$, which is $\F_{T,s}$-measurable), and since
\begin{align*}
\indicate{y \in K_{T,s}} d\mu^*(y) + M_{T,s}(y) \, dy
\end{align*}
is an $\F_{T+s}$ martingale, it is also a martingale with respect to the smaller filtration $\F_{T,s}$. This proves that $\mu^*$ is the unique random measure associated to $\{B_{T,s}, \F_{T,s}; s \geq 0 \}$.
\end{proof}

From Theorem \ref{futureDecompTheorem} we easily prove the Domain Markov property of $\mu$.

\begin{corollary}[Strong Domain Markov Property of $\mu$]
Let $T$ be an $\Ft$-stopping time. Given $\Ft$, the $d$-dimensional covariant transform of $\mu$, restricted to $\Rplus \backslash K_T$, has the law of the original measure restricted to $h_T(\R_+ \backslash K_T)$.
\end{corollary}

\begin{proof}
Let $\mu^*$ be the random measure associated to $\{B_{T,s}, \F_{T,s}; s \geq 0 \}$. By Theorem \ref{futureDecompTheorem} we know that $\mu_{d, h_T} = \mu^*$ restricted to $h_T(\Rplus \backslash K_T)$, with probability one. But $B_{T,s}$ is independent of $\F_T$ by the strong Markov property of Brownian motion, and therefore $\mu^*$ is also independent of $\F_T$.
\end{proof}

\begin{remark}
The strong version of the Domain Markov property is useful to apply at stopping times $T$ for which $\gamma(T) \in \Rplus$ almost surely. In such a case it is easy to see that $h_T(\Rplus \backslash K_T) = \Rplus$ almost surely, and the result is that the $\F_T$-conditional law of $\mu_{d, h_T}$ is the same as the original law of $\mu$ on \textit{all} of $\Rplus$ (no restriction required). In this sense these types of stopping times are renewal times for the boundary measure.
\end{remark}

We can also restate the Domain Markov property in an alternative but equivalent way.

\begin{corollary}[Alternative Statement of Domain Markov Property of $\mu$]\label{alternateDM}
For any $\Ft$-stopping time $T$, the $\F_T$-conditional law of the measure-valued process $\indicate{x \in K_{T+t}} d\mu(x), t \geq 0,$ is the law of the process
\begin{align*}
\indicate{x \in K_T} d\mu(x) + \indicate{h_T(x) \in K_t^*} \leb{h_T'(x)}^{-d} d \mu^*(h_T(x))
\end{align*}
where $\mu^*$ and $K_t^*$ are independent copies of $\mu$ and $K_t$, respectively.
\end{corollary}

\begin{remark}
Corollary \ref{alternateDM} can be thought of in the following way: suppose we have observed the SLE hull and the corresponding measure up to the stopping time $T$. Given that, the law of the part of $\mu$ that is generated after time $T$ can be realized by taking the measure $\mu^*$ corresponding to a new and independent SLE hull $K_t^*$ in the upper half plane, and then transforming (in the $d$-dimensional covariant way) the restriction of $\mu^*$ to $h_T(\Rplus \backslash K_T)$ back to $\Rplus \backslash K_T$ via $h_T^{-1}$.
\end{remark}

\subsection{General Properties}

In this section we show that the boundary measure satisfies all the properties of Theorem \ref{mainTheorem}.

\begin{proposition} \label{A_tIncrease}
Fix a closed interval $I \subset \Rplus$. Then the process $\mum{I \cap K_t}$ is flat on the open set of times $\{ t \geq 0 : \gamma(t) \not \in I \}$.
\end{proposition}

\begin{proof}
On any open interval of time for which $\gamma(t) \not \in I$ the set of swallowed points $I \cap K_t$ is not increasing, hence the same can be said of the process $\mum{I \cap K_t}$.
\end{proof}

\begin{corollary} \label{ExpectedMeasure}
For an interval $I \subset \Rplus$ we have
\begin{align*}
\expect{\mum{I \cap K_t}} = \int_I \int_{x^2/2t}^{\infty} \frac{x^{-\beta}}{\Gamma(3a - 1/2)} u^{3a-3/2} e^{-u} \, du \, dx.
\end{align*}
\end{corollary}

\begin{proof}
Since $\indicate{x \in K_t} \, d\mu(x) + M_t(x) \, dx$ is a martingale we have
\begin{align*}
\expect{\mum{I \cap K_t} + \int_I M_t(x) \, dx} = \expect{\mum{I \cap K_0} + \int_I M_0(x) \, dx} = \int_I x^{-\beta} \, dx.
\end{align*}
Therefore
\begin{align*}
\expect{\mum{I \cap K_t}} = \int_I x^{-\beta} - \expect{M_t(x)} \, dx.
\end{align*}
Substituting in the exact expression \eqref{martExpect} for $\expect{M_t(x)}$ completes the proof.
\end{proof}

\begin{remark}
Note that $\expect{\mum{I}} = \expect{\mum{I \cap K_{\infty}}} = \int_I x^{-\beta} \, dx$.
\end{remark}

\begin{lemma}[Scaling property of $\mu$]\label{muScaling}
For any $r > 0$, the random measure $\mu(r \cdot)$ has the same law as the random measure $r^d \mu(\cdot)$.
\end{lemma}

\begin{proof}
Using the Bessel scaling relation $h_t(x) \equiv r^{-1} h_{r^2t}(rx)$ of Proposition \ref{BesselProp}(a), it is easy to verify that the process $\{ M_t(x); t \geq 0, x > 0\}$ has the same law as
\begin{align*}
\left \{ r^{\beta} M_{r^2t}(rx); t \geq 0, x > 0 \right \}
\end{align*}
for any $r > 0$. Hence for any interval $I \subset \Rplus$,
\begin{align*}
\int_I r^{\beta} M_{r^2t}(rx) \, dx \equiv \int_I M_t(x) \, dx.
\end{align*}
Changing variables on the left yields
\begin{align*}
\int_{rI} M_{r^2t}(x) \, dx \equiv r^d \int_{I} M_t(x) \, dx.
\end{align*}
The random measure associated to the left hand side is $\mu(r \cdot)$ and the random measure associated to the right hand side is $r^d \mu(\cdot)$. This completes the proof.
\end{proof}

\begin{lemma}\label{chargesZero}
Let $J_{\epsilon} := (0, \epsilon)$. Then
\begin{align*}
\prob{\mum{J_{\epsilon}} > 0 \textnormal{ for all } \epsilon > 0} = 1.
\end{align*}
\end{lemma}

\begin{proof}
The event $\{ \mum{J_{\epsilon}} > 0 \}$ is $\F_{T_{\epsilon}}$-measurable, hence
\begin{align*}
\left \{ \mum{J_{\epsilon}} > 0  \textnormal{ for all } \epsilon > 0 \right \} = \bigcap_{\epsilon > 0} \left \{ \mum{J_{\epsilon}} > 0 \right \}
\end{align*}
is $\F_{0+}$-measurable. By Blumenthal $0$-$1$ law, the latter event has probability either zero or one. Hence it is enough to show that there exists a $\rho > 0$ such that $\prob{\mum{J_{\epsilon}}} \geq \rho$ for every $\epsilon > 0$. By the scaling rule of Lemma \ref{muScaling} we have
\begin{align*}
\prob{\mum{J_{\epsilon}} > 0} = \prob{\epsilon^{-d} \mum{J_1} > 0} = \prob{\mum{J_1} > 0},
\end{align*}
and clearly $\prob{\mum{J_1} > 0} > 0$ since $\expect{\mum{J_1}} > 0$.
\end{proof}

\begin{proposition}\label{massToHit}
Let $I \subset \Rplus$ be an open interval. Let $\tau_I := \left \{ t \geq 0 : \gamma(t) \in I \right \}$. Then
\begin{align*}
\prob{\mum{I} > 0 \left | \tau_I < \infty \right.} = 1,
\end{align*}
where $\tau_I := \inf \left \{ t \geq 0 : \gamma(t) \in I \right \}$.
\end{proposition}

\begin{proof}
Write $I = (x_1, x_2)$ with $0 < x_1 < x_2 < \infty$. First note that $\prob{\tau_I < \infty} > 0$ by equation \eqref{OneIntHittingProb}, so the conditioning is well defined. Let $\mu^*$ be the random measure associated to the Brownian motion $\{B_{\tau_I,s}, \F_{\tau_I, s}; s \geq 0 \}$. Since $\gamma(\tau_I) \in I$ on $\{ \tau_I < \infty \}$, Theorem \ref{futureDecompTheorem} says that $\dtrans{h_{\tau_I}}{\mu} = \mu^*$ for almost all $\omega \in \{ \tau_I < \infty \}$. Lemma \ref{chargesZero} says that for almost all $\omega \in \{ \tau_I < \infty \}$ we have $\mu^*((0, \epsilon)) > 0$ for every $\epsilon > 0$; in particular $\mu^*((0, h_{\tau_I}(x_2))) = \mu^*(h_{\tau_I}(I \backslash K_{\tau_I})) > 0$ since $h_{\tau_I}(x_2) > 0$. Thus
\begin{align*}
\prob{ \left. \dtrans{h_{\tau_I}}{\mu} \left( h_{\tau_I}(I \backslash K_{\tau_I}) \right) > 0 \, \right | \tau_I < \infty } = 1.
\end{align*}
But by definition of the $d$-dimensional covariant transform
\begin{align*}
\dtrans{h_{\tau_I}}{\mu} \left( h_{\tau_I}(I \backslash K_{\tau_I}) \right) &= \int_{I \backslash K_{\tau_I}} \leb{h_{\tau_I}'(x)}^d \, d\mu(x) \\
& \leq \int_{I \backslash K_{\tau_I}} \, d\mu(x) \\
& \leq \mum{I}.
\end{align*}
The first inequality follows by $\leb{h_{\tau_I}'(x)} \leq 1$ for all $x \in \Rplus \backslash K_{\tau_I}$.
\end{proof}

\begin{remark}
Proposition \ref{massToHit} proves one half of Theorem \ref{mainTheorem}(iv), namely that if the curve hits the open interval then $\mu$ assigns mass to it. On the event $\{ \tau_I = \infty \}$ where the curve misses $I$, by Corollary \ref{contToFirstHitting} the non-decreasing part of the Doob-Meyer decomposition for $X_t(I)$ is identically zero. Hence, by definition,
 $\mu$ assigns no mass to $I$ on this event.
\end{remark}

\begin{proposition}\label{atomFree}
With probability one $\mu$ is free of atoms.
\end{proposition}

\begin{proof}
Let
\begin{align*}
E := \left \{ \omega \in \Omega : \mum{x}(\omega) > 0 \textnormal{ for some } x \in \R_+ \right \}
\end{align*}
and suppose that $\prob{E} > 0$. Let $I_j = [2^j, 2^{j+1}]$ for $j \in \Z$. Then for each $\omega \in E$ there is a $j \in \Z$ such that the process $\mum{I_j \cap K_t}$ jumps at the time $T_x$ that $x$ joins the SLE hull. But this contradicts that for any fixed interval $I$ the process $\mum{I \cap K_t}$ is almost surely continuous in $t$.
\end{proof}

\subsection{The Boundary Measure on Smooth Domains}

In this section we define the boundary measure on an arbitrary simply connected domain $D$ with smooth boundary. Since SLE can be defined on $D$ by conformally mapping all the curves from $\H$ into $D$, we expect that we can transform the boundary measure from $\R$ to $\partial D$ in a similar way. The natural transformation to use is the $d$-dimensional covariant transformation; the last proposition shows why.

\begin{definition}\label{domainDef}
Let $D$ be a simply connected domain (other than $\C$ itself) and $w_1, w_2 \in \partial D$ be two distinct boundary points. Let $\phi : \H \to D$ be any conformal map taking $\H$ onto $D$ such that $\phi(0) = w_1, \phi(\infty) = w_2$, and assume that $\partial D$ is smooth enough that $\phi'$ extends continuously to all of $\R_+$. Let $\partial D_+$ be the boundary arc from $w_1$ to $w_2$ that is the image of $\Rplus$ (i.e. such that the interior of $D$ is to the left of $\partial D_+$). Given a Brownian motion $\{B_t; \Ft \}$, let $\mu$ be the corresponding random measure of Theorem \ref{decompTheorem}. Then we define the boundary measure (on $\partial D_+$) for the triple $(D, w_1, w_2)$ to be $\dtrans{\phi}{\mu}$, i.e. the $d$-dimensional covariant transform of $\mu$ by $\phi$.
\end{definition}

\begin{remark}
Note we are defining a boundary measure for $(D, w_1, w_2)$ as a random measure on the same probability space $(\Omega, \F, P)$ simply by mapping a random measure for $(\H, 0, \infty)$. In this sense we should think of $\mu_{d,\phi}$ as the random measure associated to the curves (and hull)
\begin{align*}
\gamma^{\phi}(t) = \phi(\gamma(t)), \quad K_t^{\phi} = \phi(K_t).
\end{align*}
Clearly $\gamma^\phi$ is a curve in $D$ going from $w_1$ to $w_2$, and by conformal invariance it has the law (under $P$) of SLE in $D$ from $w_1$ to $w_2$. We do it this way so that we are implicitly working with the same Brownian motions and filtrations, which we have already assumed to have enough nice properties to apply the Doob-Meyer decomposition.
\end{remark}

\begin{remark}
The above definition is somewhat ambiguous in that for a given triple $(D, w_1, w_2)$ there is no unique choice of the map $\phi$. Indeed, if $\phi$ satisfies the conditions of Definition \ref{domainDef} then so does $\tilde{\phi}(z) = \phi(rz)$ for any $r > 0$. However the scaling rule of Proposition \ref{muScaling} implies that $\dtrans{\phi}{\mu}$ and $\dtrans{\tilde{\phi}}{\mu}$ have the same law. To see this it is enough to consider $D = \H$, $\phi(z) = z$, and $\tilde{\phi}(z) = rz$. Then $\dtrans{\phi}{\mu} = \mu$ and
\begin{align*}
\dtrans{\tilde{\phi}}{\mu}(rI) = \int_I r^d \, d\mu(x) = r^d \mum{I} \equiv \mum{rI},
\end{align*}
with the last equality in law following from the scaling rule. Since this holds for all intervals $I \subset \Rplus$ we get $\dtrans{\tilde{\phi}}{\mu} \equiv \dtrans{\phi}{\mu}$.
\end{remark}

\begin{remark}
Since we are assuming that $\phi'$ extends continuously to all of $\R_+$, the $d$-dimensional covariant transform $\dtrans{\phi}{\mu}$ is well defined on all of $\partial D_+$. Definition \ref{transformDefn} also handles the case that $\phi'$ extends continuously only to some intervals of $\R_+$, but for the sake of exposition we decided not to use that in this section.
\end{remark}

We conclude with a characterization of $\mu_{d, \phi}$ that is the analogue of Theorem \ref{decompTheorem}. To do this we need to find the local martingale $M_t^{\phi}$ that corresponds to $D$. The ideas is that for a point $w \in \partial D_+$, the local martingale $M_t^{\phi}(w)$ should describe the conditional probability that $w$ is hit by the curve, given $\gamma^{\phi}[0,t]$. Let $S \subset \partial D_+$ be a boundary segment containing $w$. Note we may write $S = \phi(I)$ for some interval $I \subset \Rplus$ and $w = \phi(x)$ for some $x \in I$. Then
\begin{align*}
\lim_{S \downarrow \{w\}} \frac{\prob{\gamma^{\phi} \cap S \neq \emptyset \left| \gamma^{\phi}[0,t] \right. }}{\leb{S}^{\beta}} = \lim_{I \downarrow \{x\}} \frac{\prob{\gamma \cap I \neq \emptyset \left| \gamma[0,t] \right.}}{\leb{I}^{\beta}} \frac{\leb{I}^{\beta}}{\leb{S}^{\beta}} = M_t(x) \leb{\phi'(x)}^{-\beta}
\end{align*}
The first equality is by conformal invariance and the second is by equation \eqref{MtAsCondProb}.
Hence for $w \in S$ define
\begin{align*}
M_t^{\phi}(w) := \leb{\phi'(\phi^{-1}(w))}^{-\beta} M_t(\phi^{-1}(w)).
\end{align*}
Since $M_t^{\phi}(w)$ is just a rescaling of $M_t(\phi^{-1}(w))$ by a fixed constant, it follows that $M_t^{\phi}$ inherits most of the properties of $M_t$ from Section \ref{MtSection}. In particular it is a positive local martingale.

\begin{proposition}
$\mu_{d, \phi}$ is the unique random measure on $\partial D_+$ such that $\mu_{d, \phi}(\cdot \cap K_t^{\phi})$ is predictable and
\begin{align*}
\indicate{w \in K_t^{\phi}}d\mu_{d, \phi}(w) + M_t^{\phi}(w) \, dl(w)
\end{align*}
is a martingale with respect to $P$ and $\Ft$. Here $dl(w) = \leb{\phi'(\phi^{-1}(w))} dx$ is the length element on $S$.
\end{proposition}

\begin{proof}
By definition we have $M_t^{\phi}(w) \, dl(w) = \leb{\phi'(x)}^{-\beta} M_t(x) \leb{\phi'(x)} \, dx = \leb{\phi'(x)}^d M_t(x) \, dx$, where $w = \phi(x)$. By Corollary \ref{decompLemma}, $\leb{\phi'(x)}^d d \mu(x)$ is the unique random measure whose restriction to $K_t$ is predictables and that yields a martingale when added to $\leb{\phi'(x)}^d M_t(x) \, dx$. But $\leb{\phi'(x)}^d d \mu(x)$ is precisely $d \dtrans{\phi}{\mu}(w)$, by definition of the $d$-dimensional covariant transform.
\end{proof}

\section{Approximations of the Measure \label{approxSection}}

The construction of the measure in Section \ref{ConstructSection} was very much inspired by \cite{schramm_zhou:dimension}, in which Schramm and Zhou find a lower bound on the Hausdorff dimension of $\hitpoints_+$.  They used the local martingales $M_t(x)$ to construct random measures $\ep{\mu}$ on the sets $\ep{C}$, and then use a standard technique to take a $\gamma$-dependent subsequential limit of these $\ep{\mu}$ measures, on an event of positive probability, to obtain a measure on $C$. In Section \ref{prev_appearance} we briefly describe the construction that Schramm and Zhou use. In Section \ref{epSection} we will show how Theorem \ref{DellMTheorem} can be used to construct a better limiting measure for the $\ep{\mu}$, although it turns out to be the same as the one we constructed in Section \ref{ConstructSection}. In Section \ref{FrostmanSection} we use the $\ep{\mu}$ measures to show that the boundary measure $\mu$ is actually a Frostman measure on $\hitpoints_+$, and Section \ref{conformalMinkSection} contains a brief discussion on the (conjectured) relationship between the boundary measure and the Minkowski measure.

\subsection{The Schramm-Zhou Measures \label{prev_appearance}}

Schramm and Zhou begin by defining the random measures $\ep{\mu}$ on $\Rplus$ by
\begin{align*}
d \ep{\mu}(x) := \epsilon^{-\beta} \indicate{x \in C^{\epsilon}} \, dx.
\end{align*}
Given $\delta > 0$ and an interval $I \subset \Rplus$, they use the bound of Corollary \ref{IndicatorBound} to prove that there exists a constant $R_{\delta, I} < \infty$ such that
\begin{align*}
\sup_{\epsilon > 0} \expect{\Henergy_{d - \delta} \left( \ep{\mu}, I \right)} \leq R_{\delta, I},
\end{align*}
where $\Henergy_{\alpha}(\nu, I)$ is the $\alpha$-\textit{energy} of a measure $\nu$ (restricted to $I$) defined by
\begin{align*}
\Henergy_{\alpha} \left( \nu, I \right) := \int_I \int_I \frac{d \nu (x) d \nu (y)}{|x-y|^{\alpha}}.
\end{align*}
From this, they prove the existence of a $\lambda > 0$ and an event $E$ with $\prob{E} > \lambda$ and the following additional property: for every $\omega \in E$, there is a subsequence $\epsilon_j(\omega)$ tending to zero such that $\mu^{\epsilon_j(\omega)}(\cdot)(\omega)$ converges weakly to a measure $\mu^o(\cdot)(\omega)$ supported on $C(\omega)$, and that $\mu^o(\cdot)(\omega)$ further satisfies
\begin{align*}
\mu^o(I)(\omega) > \lambda, \, \Henergy_{d - \delta} \left( \mu^o(\cdot)(\omega), I \right) < 1/\lambda.
\end{align*}
The limiting measure $\mu^o(\cdot)(\omega)$ for $\omega \in E$ is enough to get a lower bound on the Hausdorff dimension of $C(\omega)$, but is hardly a satisfactory candidate as a natural measure on $\hitpoints_+$. Indeed, it only exists on an event of positive probability, and even then only as a subsequential limit that depends on the $\omega$ in question. Using Theorem \ref{DellMTheorem} we prove the much stronger result that the sequence of random measures $\mu^{\epsilon_j}$ converges weakly, with probability one, to the random measure $\mu$ along some fixed subsequence $\epsilon_j$ that tends to zero. This is the subject of the next section.

\subsection{Convergence of the Schramm-Zhou Measures \label{epSection}}

Like $\mu$, the $\ep{\mu}$ measures appear naturally as the non-decreasing part of the Doob-Meyer decomposition of a supermartingale. Fix an interval $I \subset \Rplus$, and for a fixed $\epsilon > 0$ consider the process
\begin{align*}
\epsilon^{-\beta} \left | I \cap \te{C} \right|.
\end{align*}
As $t \to \infty$ this clearly increases to $\ep{\mu}(I)$. Using that $\te{M}(x) = \epsilon^{-\beta}$ if $x \in \te{C}$, we may also write
\begin{align*}
\epsilon^{-\beta} \left | I \cap \te{C} \right| = \int_I \te{M}(x) \indicate{x \in \te{C}} \, dx.
\end{align*}
Observe that this term shows up in the decomposition
\begin{align}\label{DiscreteDecomp}
\int_I \te{M}(x) \indicate{x \not \in \te{C}} \, dx = \int_I \te{M}(x) \, dx - \int_I \te{M}(x) \indicate{x \in \te{C}} \, dx.
\end{align}
Equation \eqref{DiscreteDecomp} is clearly the Doob-Meyer decomposition of the supermartingale
\begin{align*}
\te{X}(I) := \int_I \te{M}(x) \indicate{x \not \in \te{C}},
\end{align*}
with $\epsilon^{-\beta} \left| I \cap \te{C} \right|$ being the non-decreasing part. In \eqref{DiscreteDecomp}, it is not immediately clear that either one of the two terms on the right hand side actually converges as $\epsilon \downarrow 0$, but fortuitously the $\te{X}(I)$ term does. Indeed, for a fixed $t$ note that
\begin{align*}
\te{M}(x) \indicate{x \not \in \te{C}} = M_t(x) \indicate{x \not \in \te{C}} \leq M_t(x).
\end{align*}
for all $\epsilon > 0$. Since the sets $\te{C}$ decrease with $\epsilon$ to $C_t$, we have that $M_t(x) \indicate{x \not \in \te{C}}$ increases pointwise to $M_t(x) \indicate{x \not \in C_t}$ as $\epsilon \downarrow 0$. But
\begin{align*}
M_t(x) - M_t(x) \indicate{x \not \in C_t} = M_t(x) \indicate{x \in C_t},
\end{align*}
and
\begin{align*}
\int_I M_t(x) \indicate{x \in C_t} \, dx = 0,
\end{align*}
since $\haussdim{C_t} \leq \haussdim{C} < 1$. Therefore
\begin{align*}
\int_I M_t(x) \, dx = \int_I M_t(x) \indicate{x \not \in C_t} \, dx.
\end{align*}
Consequently, by monotone convergence,
\begin{align*}
\lim_{\epsilon \downarrow 0} \te{X}(I) = \lim_{\epsilon \downarrow 0} \int_I M_t(x) \indicate{x \not \in \te{C}} \, dx = \int_I M_t(x) \indicate{x \not \in C_t} \, dx = \int_I M_t(x) \, dx = X_t(I).
\end{align*}
Hence $\te{X}(I)$ almost surely increases pointwise to $X_t(I)$ as $\epsilon \downarrow 0$. In Section \ref{ConstructSection} we showed that $X_t(I)$ is of \textit{class $\mathcal{D}$} and regular, so by Theorem \ref{DellMTheorem} we have that the increasing part of $\te{X}(I)$ converges in $L^1$ to the increasing part of $X_t(I)$ as $\epsilon \downarrow 0$. This gives the following

\begin{proposition}\label{DellMApplication}
Fix an interval $I \subset \Rplus$. Then for all stopping times $T$,
\begin{align*}
\lim_{\epsilon \downarrow 0} \expect{\leb{ \epsilon^{-\beta} \leb{I \cap C_{T}^{\epsilon}} - A_T(I) }} = 0
\end{align*}
\end{proposition}

From this we immediately deduce the following:

\begin{proposition}\label{weakConvergence}
There exists a fixed sequence $\epsilon_j \downarrow 0$ such that $\mu^{\epsilon_j}$ converges weakly to the random measure $\mu$, with probability one.
\end{proposition}

\begin{proof}
It suffices to show that $\mu^{\epsilon_j}(I)$ converges almost surely to $\mu(I)$ for every $I \in \mathcal{Q}$. For any interval $I \subset \Rplus$ we clearly have
\begin{align*}
\epsilon^{-\beta} \leb{I \cap C_{T_I}^{\epsilon}} = \ep{\mu}(I).
\end{align*}
Hence it follows from Proposition \ref{DellMApplication} that $\ep{\mu}(I)$ converges in $L^1$ to $\mu(I)$. Since $\mathcal{Q}$ is countable, it follows from a diagonalization argument that there is a common, fixed subsequence $\epsilon_j$ tending to zero along which $\mu^{\epsilon_j}(I)$ converges almost surely to $\mu(I)$ for every interval $I \in \mathcal{Q}$.
\end{proof}

\subsection{The Measure as a Frostman Measure \label{FrostmanSection}}

Using the convergence of the last section, we prove in this section that $\mu$ is almost surely measure a Frostman measure on the set $C$. This result is not surprising since, as mentioned in Section \ref{prev_appearance}, a Frostman measure was already constructed in \cite{schramm_zhou:dimension} as the limit of the $\ep{\mu}$ measures (on some event of positive probability). We are able to prove the following:

\begin{proposition}\label{ConformalFrostman}
Let $I = (x_1, x_2]$ with $0 < x_1 < x_2 < \infty$. Then for every $\delta > 0$ there exists a constant $R_{\delta} > 0$ (also depending on $I$ and $\kappa$) such that the expected $(d-\delta)$-energy of $\mu$ restricted to $I$ is finite, i.e.
\begin{align*}
\expect{\int_I \int_I \frac{d \mu(x) d \mu(y)}{|x-y|^{d-\delta}}} \leq R_{\delta}.
\end{align*}
Consequently, $\mu$ is a Frostman measure on $C$.
\end{proposition}

\begin{proof}
Using Proposition \ref{IndicatorBound}, we may bound the expected $(d-\delta)$-energy by
\begin{align*}
\expect{\Henergy_{d-\delta} \left( \ep{\mu}, I \right)} &= \expect{\int_I \int_I \frac{\epsilon^{-2\beta} \indicate{x,y \in \ep{C}}}{|x-y|^{d-\delta}} \, dx dy} \\
&\leq c \int_I \int_I \frac{dx \, dy}{(x \wedge y)^{\beta}|x-y|^{1-\delta}} \\
&\leq c x_1^{-\beta} \int_I \int_I \frac{dx \, dy}{|x-y|^{1-\delta}} \\
&= cx_1^{-\beta} \frac{|I|^{1+\delta}}{\delta (\delta + 1)}.
\end{align*}
Note that the bound is independent of $\epsilon$. Since, by Proposition \ref{weakConvergence}, $\mu^{\epsilon_j}$ converges weakly to $\mu$, it follows that
\begin{align*}
\expect{\Henergy_{d-\delta} \left( \mu, I \right)} &\leq \liminf_{\epsilon_j \downarrow 0} \expect{ \Henergy_{d-\delta} \left( \mu^{\epsilon_j} , I \right) } \leq R_{\delta} < \infty,
\end{align*}
which completes the proof.
\end{proof}

\subsection{The Conformal Minkowski Measure \label{conformalMinkSection}}

As we mentioned in the introduction, it remains an open problem as to whether or not our measure is related to the Minkowski measure of $\hitpoints_+$, or perhaps is some sort of Hausdorff content. Using the measures $\ep{\mu}$ we can show that our measure is some sort of \textit{Conformal Minkowski measure}. The usual definition for the $d$-dimensional Minkowski measure of an interval $I$ would be the limit
\begin{align*}
\lim_{\epsilon \downarrow 0} \epsilon^{-\beta} \left| \left \{ x \in I : \dist(x, \gamma) < \epsilon \right \} \right|.
\end{align*}
It is not easy to prove that this limit exists. It is difficult to obtain statistics on the Euclidean distance from a fixed point to an SLE curve, and even if it were possible it still might not be the right quantity to consider since Euclidean distance is not invariant under conformal maps. However, Lemma \ref{ctDistance} shows that $M_t(x)^{-\beta}$ acts as a natural distance from $x$ to the curve, and that $x \in C^{\epsilon}$ is (almost) equivalent to $\dist(x, \gamma) \leq 4 \epsilon$. This motivates the consideration of
\begin{align}\label{ConformalMinkowski}
\lim_{\epsilon \downarrow 0} \epsilon^{-\beta} \left| I \cap C \right|,
\end{align}
which we call the \textit{Conformal Minkowski measure} of $I$. Theorem \ref{DellMTheorem} and the last section show that the limit exists, and in fact converges to $\mu(I)$ in $L^1$. Based on this, we could alternatively call $\mu$ the Conformal Minkowski measure instead. In fact we conjecture, although have been unable to prove, that $\mu(I)$ is exactly the Minkowski measure of $I$, up to a fixed multiplicative constant that is the same for all $I$.

\bigskip
\bigskip

\noindent \textbf{Acknowledgements:} We are grateful to Greg Lawler for several very helpful ideas about the construction of the measure, and for giving us access to early versions of \cite{lawler:curvedim}. We thank Marc Yor for providing the extremely helpful references \cite{dellacherie_meyer} and \cite{yor_expfuncs}, for helpful comments on an earlier draft of this paper, and for many enlightening conversations. We also thank S.R.S. Varadhan for repeatedly explaining the fine points of the Doob-Meyer decomposition to the first author, and for earlier discussions on the convergence of the $\ep{\mu}$ measures. The first author would like to thank Vladas Sidoravicius and Wendelin Werner for hosting him in spring 2008 at IMPA and the \'{E}cole Normale Sup\'{e}rieure, respectively, where some of this work was completed. Most of it was done at the Courant Institute of Mathematical Sciences at New York University as part of the first author's PhD thesis, under the supervision of the second author. We would like to thank the Courant Institute and its faculty and staff for being supportive of our research.

\bigskip

At the time of Oded Schramm's untimely passing we were nearly finished with this paper, and we were looking forward to sending it to him and receiving back the helpful comments and kind words of encouragement that he was so known for. We are thankful to have known him, and we will truly miss him.

\bibliographystyle{alpha}
\bibliography{C:/Academic/Bibliographies/SLE/SLE}

\end{document}